\documentclass[10pt,journal]{IEEEtran}
\ifCLASSOPTIONcompsoc
  \usepackage[nocompress]{cite}
\else
  \usepackage{cite}
\fi

\usepackage[english]{babel}
\usepackage{amsmath,amsthm}
\usepackage{amsfonts}
\usepackage{dsfont}
\usepackage[comma,numbers,square,sort&compress]{natbib}
\usepackage{graphicx,subfigure,amsmath,amssymb,mathrsfs,amsfonts,amstext,amsthm}
\usepackage{latexsym, amssymb}
\usepackage{graphicx}
\usepackage{subfigure}
\usepackage{pgf,tikz}
\usepackage{pstricks}

\newtheoremstyle{mythm}{1.5ex plus 1ex minus .2ex}{1.5ex plus 1ex minus .2ex} {\rm}{\parindent}{\it\it}{\rm{:}}{1em}{}
\theoremstyle{mythm}

\renewcommand{\IEEEQED}{\IEEEQEDopen}

\newtheorem{thm}{Theorem}[section]
\newtheorem{cor}[thm]{Corollary}
\newtheorem{lem}[thm]{Lemma}

\theoremstyle{definition}
\newtheorem{defn}[thm]{Definition}
\newtheorem{example}[thm]{Example}
\theoremstyle{remark}
\newtheorem{rem}[thm]{Remark}
\numberwithin{equation}{section}
\begin{document}

\title{Heterogeneous Hegselmann-Krause Dynamics with Environment and Communication Noise}

 \author{Ge~Chen,~\IEEEmembership{Senior Member,~IEEE}, Wei Su,~Songyuan Ding,~Yiguang Hong,~\IEEEmembership{Fellow,~IEEE}
\IEEEcompsocitemizethanks{\IEEEcompsocthanksitem
This research was supported by the National Natural Science Foundation of China under grants No. 11688101, 91427304, 61733018, 61803024, and 61573344, and the National Key Basic Research Program of China (973 program) under grant 2016YFB0800404, and the Fundamental Research Funds for the Central Universities under grant No. FRF-TP-17-087A1.
\IEEEcompsocthanksitem
Ge Chen is with the National Center for Mathematics and Interdisciplinary Sciences \& Key Laboratory of Systems and
Control, Academy of Mathematics and Systems Science, Chinese Academy of Sciences, Beijing 100190,
China, {\tt
chenge@amss.ac.cn}
\IEEEcompsocthanksitem
Wei Su is with School of Automation and Electrical Engineering, University of Science and Technology Beijing, Beijing 100083, China, {\tt
suwei@amss.ac.cn}
\IEEEcompsocthanksitem
Songyuan Ding is with the Collaborative Innovation Centre of Chemistry for Energy Materials (iChEM), and Department of Chemistry, College of Chemistry and Chemical Engineering, Xiamen University, Xiamen 361005,  China, {\tt syding@xmu.edu.cn}
\IEEEcompsocthanksitem
Yiguang Hong is with the Key Laboratory of Systems and
Control, Academy of Mathematics and Systems Science, Chinese Academy of Sciences, Beijing 100190, {\tt yghong@iss.ac.cn}
}}

\IEEEtitleabstractindextext{%
\begin{abstract}
The Hegselmann-Krause (HK) model is a well-known opinion dynamics, attracting a significant amount of interest from a number of fields. However, the heterogeneous HK model is difficult to analyze - even the most basic property of convergence is still open to prove. For the first time, this paper takes into consideration heterogeneous HK models with environment or communication noise. Under environment noise, it has been revealed that the heterogeneous HK model with or without global information has a phase transition for the upper limit of the maximum opinion difference, and has a critical noise amplitude depending on the minimal confidence threshold for quasi-synchronization. In addition, the convergence time to the quasi-synchronization is bounded by a negative exponential distribution. The heterogeneous HK model with global information and communication noise is also analyzed. Finally, for the basic HK model with communication noise, we show that the heterogeneous case exhibits a different behavior regarding quasi-synchronization from the homogenous case. Interestingly, raising the confidence thresholds of constituent agents may break quasi-synchronization. Our results reveal that the heterogeneity of individuals is harmful to synchronization, which may be the reason why the synchronization of opinions is hard to reach in reality, even within that of a small group.
\end{abstract}

}

\maketitle

\IEEEdisplaynontitleabstractindextext
\IEEEpeerreviewmaketitle

\ifCLASSOPTIONcompsoc
\IEEEraisesectionheading{\section{Introduction}\label{intro}}
\else
\section{Introduction}\label{intro}
\fi

People present their opinions on certain events in which it is necessary for a group to reach shared decisions. Agreement (also known as consensus or synchronization) is one of the most important aspects of social group dynamics, making a position stronger and amplifying its impact on society. It is very natural for individuals to have different opinions on the same event, though it is rather complex to study the dynamics of how a group reaches an agreement. Opinion dynamics is a research field in which various tools are used to study the dynamical processes of the formation, diffusion, and evolution of public opinion. In fact, opinion dynamics has been an important issue of research in sociology and has also attracted a lot of attention in recent years from many other disciplines such as physics, mathematics, computer science, social psychology, and philosophy \cite{Lorenz2007,Sirbu2017}.

The study of opinion dynamics can be traced back to
the two-step flow of communication model studied by Lazarsfeld and Katz in the 1940-50s \cite{Katz1955}. This model posits that most people form their opinions under the influence of opinion leaders, who, in turn, are influenced by the mass media. Another famous early work on opinion dynamics is the social power model proposed by French \cite{French1956}.
 Based on a discussion and classification of ``social power", this model describes the diffusion of social influence and the formation of public opinions in social networks. The French model is a special case of the model proposed by DeGroot \cite{DeGroot1974}; as such, they are referred to as the ``French-DeGroot model" in some works  \cite{Proskurnikov2017}.
 Later, some new theories of opinion dynamics have been developed, namely the social
influence network theory \cite{Friedkin1998}, social impact theory \cite{Latane1981}, and dynamic social impact theory \cite{Latane1996}.
Recently, bounded confidence (BC) models of opinion dynamics has been of interest. The BC models adopt a mechanism where one individual is not willing to accept the opinion of another one if he/she feels their opinions have a big gap. One well-known BC model was formulated by Hegselmann and Krause \cite{hegselmann2002}, called the Hegselmann-Krause (HK) model, where all agents synchronously update their opinions by averaging the opinions in their confidence bounds. Another well-known BC model was proposed by
Deffuant \emph{et al.} \cite{Deffuant2000}, which is similar to the HK model, though it instead employs a pairwise-sequential updating procedure in place of the synchronized one within the HK model. For opinion dynamics research, one core issue is whether and how agreement can be reached.

\paragraph*{Motivations}
Among a wide variety of opinion dynamics models, the HK model is a particularly interesting one that has attracted a lot of attention in recent years. Because the inter-agent topology of the HK model is time-varying and determined by the agents' states, whereas the agents' states depend on the topology, the theoretical analysis of the HK model is difficult. The current analysis of the HK model focuses on the most basic property of convergence. For the homogeneous HK model - where all agents have the same confidence bound - the convergence and convergence rate have been well studied \cite{Blondel2009,Blondel2010,Touri2011,Chazelle2011,Bhattacharyya2013}.
Also, there exists some theoretical research on varieties of the homogeneous HK model, such as the systems with decaying confidence \cite{Morarescu2011},
with distance-dependent interaction weight \cite{Motsch2014}, or with continuous agents \cite{Wedin2015}.
For the heterogeneous HK model - where the confidence bounds of the agents can be different -Su \emph{et al.} \cite{Su2017}
prove that partial agents in the system  will reach static states after a finite time, however cannot prove the convergence of other agents. Besides, the opinions of all agents are shown to be convergent if each agent maintains communication with others during a long enough period of time \cite{Etesami2015}, or if the confidence bound of each agent is either $0$ or $1$  \cite{Chazelle2016}. However, the convergence of the general heterogeneous HK model without additional conditions is still an open problem (Conjecture 2.1 in \cite{Mirtabatabaei2012}), despite it having been supported by simulations \cite{Lorenz2010}. The analysis of the heterogeneous HK model is particularly important since there are always differences between individuals, contributing to one motivation of this paper.

Another motivation of this paper is to study the collective behavior of the HK model affected by noise. There is a consensus that all natural systems are inextricably affected by noise \cite{Sagues2007}.  Actually, how noise affects the collective behavior of a complex system has garnered considerable interest from researchers and developers in differing fields. Generally, the noise in engineering systems may break their ordered structures, in which case one wishes to reduce the effect of the noise. However, in many natural and social systems, the noise may drive the systems to produce ordered structure \cite{Sagues2007,Shinbrot2001}.
As a matter of fact, the actual opinions of individuals are inevitably influenced by the randomness during opinion transmission and evolution, which could be attributable to the many exogenous
factors like T.V., blogs, and newspapers, or the communication between individuals. Thus, it has been recognized by several studies that randomness is an essential factor for the investigation of opinion dynamics in reality
\cite{Mas2010,Pineda2011,Grauwin2012,Carro2013, Pineda2013,Su2017b,Wang2015}.
In many studies, an interesting phenomenon was found where the noise in some situations could play a positive role in enhancing the synchronization or reducing the disagreement of opinions. Yet almost all of these findings were based on simulations, and the theoretical analysis is limited.
 In our previous paper, a homogeneous HK model with environment noise was studied \cite{Su2017b},
 but the analysis method cannot be applied to the heterogeneous case. This paper will analyze the heterogeneous HK model with environment noise by means of a completely different method. Also, to be more practical, this paper considers different types of heterogeneous HK models that may be affected by communication noise and global information. The communication noise is caused by individuals potentially not expressing their own opinion or not accurately understanding the opinions of others, while the global information denotes the background opinion modeled by the average opinion of all individuals.

\paragraph*{Contributions}
Based on the above motivations, for the first time, this paper analyzes heterogeneous HK models with or without global information, and with environment or communication noise. We show in detail that, under environment noise, the HK model with or without global information has a phase transition for the upper limit of the maximum opinion difference, as well as a critical noise amplitude for quasi-synchronization. The critical noise amplitude only depends on the minimal confidence threshold among all individuals. Also, it is shown that the convergence time to quasi-synchronization is bounded by a negative exponential distribution.

For the HK model with small communication noise, quasi-synchronization can be still reached if it contains global information. However, if it does not contain global information, the heterogeneous model exhibits a different behavior regarding quasi-synchronization from the homogenous model. Interestingly, raising the confidence thresholds of constituent agents may break quasi-synchronization. The above results reveal that the heterogeneity is harmful to synchronization, which may be the reason why it is challenging to reach the synchronization of opinions in reality, even within a small group. It is worth noting that the heterogeneous HK model without noise is hard to analyze. Even so, we provide some exact properties for when noise is considered.

\paragraph*{Organization}
 Section \ref{Sec_Def} gives the
preliminaries and then formulates our problem, while Sections \ref{Analysis_sec_1} and \ref{Analysis_sec_3}
present our main results with strict analysis.   Finally, Section \ref{Conclusions} concludes this paper.

\section{Models and Definitions}\label{Sec_Def}
\renewcommand{\thesection}{\arabic{section}}

The original HK model assumes that there are $n (n\geq 3)$ individuals or agents in a group. Each individual $i$ has a time-varying opinion $x_i(t)\in [0, 1]$, and can only communicate opinions with his/her friends, which is defined by where the difference of opinions is not bigger than a confidence threshold $r_i\in (0,1]$. This mechanism is based on a practical phenomenon where one individual is not willing to accept the opinion of another if he/she feels their opinions have a large gap. Let
\begin{equation*}\label{neigh}
 \mathcal{N}_i(t)=\{1\leq j\leq n\; \big|\; |x_j(t)-x_i(t)|\leq r_i\}
\end{equation*}
denote the neighbor set of agent $i$ at time $t$. Here we note that an individual's neighbor set contains himself/herself.
The evolution of opinions of the HK model accords to the following dynamics \cite{hegselmann2002}:
\begin{equation}\label{HKnoiseless}
  x_i(t+1)=|\mathcal{N}_i(t)|^{-1}\sum_{j\in \mathcal{N}_i(t)}x_j(t), ~i=1, \ldots, n, t\geq 0,
\end{equation}
where $|S|$ denoting the cardinal number of a set $S$.

If $r_1=\cdots=r_n$, then the system (\ref{HKnoiseless}) is called the homogeneous HK model, otherwise it is referred to as the heterogeneous HK model. The HK model is a typical self-organized system that has attracted a significant amount of interest, but has shown difficult to analyze. Currently, the analysis of the HK model focuses on the homogeneous case, while the analysis of the heterogeneous case is almost lacking.

The original HK model does not consider the effect of noise. However, all actual systems are inextricably affected by noise. To be more practical, this paper will consider the HK model affected by either environment noise or communication noise.


\subsection{Heterogeneous HK Models with Environment Noise}

The dynamics of opinions in real societies is also affected by many exogenous factors such as T.V., blogs, newspapers, and so on \cite{Pineda2015}. Some HK-type systems under the effect of exogenous factors have been considered in recent years \cite{Pineda2015,SChen2016,XChen2017}. This paper considers exogenous factors as environment noises. Following previous work in noisy opinion dynamics \cite{Mas2010, Pineda2013,Su2017b}, we confine the values of an individual's opinion to the interval $[0,1]$. Let $\Pi_{[0,1]}(\cdot)$ denote the projection onto the interval $[0,1]$, i.e., for any $x\in\mathds{R}$,
\begin{eqnarray*}
\Pi_{[0,1]}(x)=\left\{%
\begin{array}{ll}
1 & \mbox{if}~x>1\\
x & \mbox{if}~x\in[0,1]\\
0 & \mbox{otherwise}
\end{array}%
\right..
\end{eqnarray*}
Consider a basic HK model with environment noise as follows:
Denote $\mathcal{V}=\{1,2,\ldots,n\}$ as the set of all agents with $n\geq 3$.
For all $i\in\mathcal{V}$ and  $t\geq 0$, let
\begin{equation}\label{m1}
  x_i(t+1)=\Pi_{[0,1]}\Big( |\mathcal{N}_i(t)|^{-1}\sum\limits_{j\in \mathcal{N}_i(t)}x_j(t)+\xi_i(t)\Big),
\end{equation}
where $\xi_i(t)\in [-\eta,\eta]$ is a bounded noise with $\eta>0$ being a constant.

A natural consideration is that all agents may be affected in reality by the background opinion. Accordingly, an interesting problem is how the background opinion affects the collective behavior of the opinions of agents. For simplicity, this paper models the background opinion as the average opinion $x_{\rm{ave}}(t):=\frac{1}{n}\sum_{i=1}^n x_i(t)$, and each agent $i$ has a belief factor $\omega_i\in(0,1)$ in the global information $x_{\rm{ave}}(t)$.
The HK model with global information and environment noise is formulated as
\begin{eqnarray}\label{m3}
&&x_i(t+1)\\
  &&=\Pi_{[0,1]}\Big( \omega_i x_{\rm{ave}}(t)+ \frac{1-\omega_i}{|\mathcal{N}_i(t)|}\sum\limits_{j\in \mathcal{N}_i(t)}x_j(t)+\xi_i(t)\Big).\nonumber
\end{eqnarray}

Let $x(t):=(x_1(t),\ldots,x_n(t))$.
To be more practical we consider a wide class of noises which contain not only independent noises but also correlated noises.
For systems (\ref{m1}) and (\ref{m3}),
let $\Omega^t=\Omega_n^t\subseteq \mathds{R}^{n\times(t+1)}$ be the sample space of $(\xi_i(t'))_{0\leq t' \leq t, i \in \mathcal{V}}$,
and $\mathcal{F}^t=\mathcal{F}_n^t$ be its Borel $\sigma$-algebra. Additionally we define $\Omega^{-1}$ be the empty set.
Let $P=P_n$ be the probability measure on $\mathcal{F}^{\infty}$ for $(\xi_i(t'))_{t' \geq 0,i\in\mathcal{V}}$, so  the
 probability space is written as $( \Omega^{\infty},\mathcal{F}^{\infty},P)$.
 We assume that the noises $\{\xi_i(t)\}$ satisfy the following assumption.

\textbf{(A1)} For any $t\geq 0$ and any states $x(0),\ldots,x(t)\in[0,1]^n$, the joint probability density of $(\xi_1(t),\ldots,\xi_n(t))$ has a positive lower bound, i.e., there exists a constant $\underline{\rho}>0$ such that
\begin{eqnarray*}\label{noise_cond_2}
\begin{aligned}
&P\left(\bigcap_{i=1}^n \left\{ \xi_i(t)\in [a_i,b_i]\right\}| \forall x(0),\ldots,x(t)\in[0,1]^n \right)\\
 &\geq  \underline{\rho} \prod_{i=1}^n (b_i-a_i)
\end{aligned}
\end{eqnarray*}
for any $t\geq 0$ and real numbers $a_i$ and $b_i$ satisfying $-\eta\leq a_i < b_i \leq \eta$.

The positive lower bound in (A1) simply means that for any $t\geq 0$, all individuals are affected by noise, and the noise has a positive probability density over $[-\eta,\eta]$. It is easy to see that if $\xi_i(t)$ is uniformly and independently distributed in $[-\eta, \eta]$, then it satisfies (A1). Some other bounded noises also satisfy (A1), such as the truncated Gaussian noise \cite{Bjorsell2007}, as well as the discrete time version of frequency fluctuations generated by sinusoidal functions and Wiener processes \cite{CaiWu2004}.

\subsection{Heterogeneous HK Models with Communication Noise}

In reality, communication between individuals may be subject to noise because individuals may not express their own opinion or accurately understand the opinions of others. The heterogeneous HK dynamics with communication noise can be formulated as
\begin{eqnarray}\label{m4}
x_i(t+1)=\Pi_{[0,1]}\Big(|\mathcal{N}_i(t)|^{-1}\sum\limits_{j\in \mathcal{N}_i(t)}[x_j(t)+\zeta_{ji}(t)] \Big)
\end{eqnarray}
for all $i\in\mathcal{V}$ and $t\geq 0$,
where $\zeta_{ji}(t)\in[-\eta,\eta]$ denotes the communication noise from agent $j$ to $i$ at time $t$, with $\zeta_{ii}(t)\equiv 0$.

Similar to (\ref{m3}) we also consider an HK model with global information and communication noise as follows:
\begin{eqnarray}\label{m6}
  &&x_i(t+1)=\Pi_{[0,1]}\Big( \omega_i x_{\rm{ave}}(t)+(1-\omega_i)|\mathcal{N}_i(t)|^{-1}\nonumber\\
  &&~~~~\cdot\sum\limits_{j\in \mathcal{N}_i(t)}[x_j(t)+\zeta_{ji}(t)]\Big),~~i\in\mathcal{V}, t\geq 0.
\end{eqnarray}


For systems (\ref{m4}) and (\ref{m6}),
let $\Omega^t=\Omega_n^t$ be the sample space of $(\zeta_{ji}(t'))_{0\leq t' \leq t, i \in \mathcal{V}, j\in\mathcal{N}_i(t')}$,
and $\mathcal{F}^t=\mathcal{F}_n^t$ be its Borel $\sigma$-algebra. Additionally we define $\Omega^{-1}$ as the empty set.
Let $P=P_n$ be the probability measure on $\mathcal{F}^{\infty}$ for $(\zeta_{ji}(t'))_{t' \geq 0,i\in\mathcal{V},j\in\mathcal{N}_i(t')}$, so  the
 probability space is written as $( \Omega^{\infty},\mathcal{F}^{\infty},P)$. Define the set of agent pairs  $\mathcal{E}(t)$ by $\mathcal{E}(t):=\{(i,j): i\in\mathcal{V},j\in \mathcal{N}_i(t)\backslash\{i\}\}$.
Similar to (A1), we assume the noises $\{\zeta_{ji}(t)\}$ satisfying the following assumption.

\textbf{(A2)} For any $t\geq 0$ and any states $x(0),\ldots,x(t)\in[0,1]^n$,  if $\mathcal{E}(t)$ is not empty, then,
 the joint probability density of $\{\zeta_{ji}(t)\}_{(i,j)\in\mathcal{E}(t)}$ has a positive lower bound, i.e., there exists a constant $\underline{\rho}>0$ such that
\begin{eqnarray*}\label{noise_cond_2}
\begin{aligned}
&P\Big(\bigcap_{(i,j)\in\mathcal{E}(t)} \big\{ \zeta_{ji}(t)\in [a_{i}^j,b_{i}^j]\big\}| \forall x(0),\ldots,x(t)\Big)\\
 &\geq  \underline{\rho}\prod_{(i,j)\in\mathcal{E}(t)} (b_i^j-a_i^j)
\end{aligned}
\end{eqnarray*}
for any $t\geq 0$ and real numbers $a_i^j$ and $b_i^j$ satisfying $-\eta\leq a_i^j< b_i^j \leq \eta$.




\subsection{Definitions}

As we know, the conventional consensus/synchronization concept signifies that the states of all agents are exactly the same, and this concept has been well studied in noise-free opinion dynamics and multi-agent systems. However, if considering a system affected by noise, the strict consensus/synchronization behavior may not be reached.

Define  $$d_{\mathcal{V}}(t):=\max\limits_{i, j\in \mathcal{V}}|x_i(t)-x_j(t)|$$ to be the maximum opinion difference at time $t$.
Let
\begin{eqnarray*}
\underline{d}_{\mathcal{V}}:=\liminf\limits_{t\rightarrow \infty}d_{\mathcal{V}}(t)~~ \mbox{and}~~\overline{d}_{\mathcal{V}}:=\limsup\limits_{t\rightarrow \infty}d_{\mathcal{V}}(t)
\end{eqnarray*}
 denote the lower limit and upper limit of the maximum opinion difference respectively.

 Similar to
\cite{Su2017}, we relax the concept of synchronization to \emph{quasi-synchronization} which is defined as follows:
\begin{defn}\label{robconsen}
We say that \emph{quasi-synchronization} is asymptotically reached if $\overline{d}_{\mathcal{V}} \leq \min_{1\leq i\leq n}r_i$.
\end{defn}

From Definition \ref{robconsen}, if quasi-synchronization is reached, then any two agents can communicate directly in the limit state. In fact, in Theorems \ref{theorem_1}, \ref{theorem_2}, \ref{theorem_2c}, and \ref{theorem_1c} below, wherever a bound $\bar{d}_{\mathcal{V}} \leq  \min r_i$ exists, almost surely there is some finite time $t_0$ such that $d_{\mathcal{V}} \leq  \min r_i$ for all $t\geq t_0$. In other words, almost surely agents communicate directly from some time onwards.

\renewcommand{\thesection}{\Roman{section}}
\section{Main Results for Systems (\ref{m1}), (\ref{m3}), and (\ref{m6})}\label{Analysis_sec_1}
\renewcommand{\thesection}{\arabic{section}}

In this section we will analyze systems (\ref{m1}) and (\ref{m3}) which exhibit a phase transition for quasi-synchronization behavior
depending on the noise amplitude $\eta$. Let $r_{\min}=\min_{1\leq i\leq n} r_i $ and $r_{\max}=\max_{1\leq i\leq n} r_i$.
Also, for any $\alpha>0$ we set
\begin{equation*}
  c_{\alpha}^1=c_{\alpha}^1(\eta):=\left\{
           \begin{array}{ll}
            1  &  \mbox{if~} \eta> \max\{ \frac{r_{\min}}{2}, \frac{r_{\max}}{n}\}\\
            2\eta-\alpha  & \mbox{otherwise} \\
           \end{array}
         \right..
\end{equation*}
With the above definitions we first give our main result for system (\ref{m1}) as follows:
\begin{thm}[Phase transition and switching interval in the heterogeneous HK model with environment noise]\label{theorem_1}
Consider system (\ref{m1}) satisfying (A1) and $r_{\min}<1$. Then for any initial state $x(0)\in [0,1]^n$,
almost surely (a.s.)
\begin{equation}\label{theo1_00}
  \overline{d}_{\mathcal{V}}\left\{
           \begin{array}{ll}
            =2\eta  & \mbox{if~} \eta\leq \frac{r_{\min}}{2} \\
            \geq  \min\{2\eta,1\}  & \mbox{if~} \eta> \frac{r_{\min}}{2} \\
            =1  &  \mbox{if~} \eta> \max\{ \frac{r_{\min}}{2}, \frac{r_{\max}}{n}\}
           \end{array}
         \right..
\end{equation}
Also, for any constant $\alpha>0$, if we set $\tau_0=0$, and $\tau_k$ to be the stopping time as
 \begin{eqnarray}\label{theo1_00_a}
\tau_k=\left\{%
\begin{array}{ll}
\min\{s>\tau_{k-1}:d_{\mathcal{V}}(s)\leq \alpha\} & \mbox{if $k$ is odd}\\
\min\{s>\tau_{k-1}:d_{\mathcal{V}}(s)\geq c_\alpha^1\} & \mbox{if $k$ is  even}
\end{array}%
\right.,
\end{eqnarray}
 then for all $k\geq 0$ and $t\geq 0$ we have
\begin{eqnarray}\label{theo1_01}
\begin{aligned}
P\left(\tau_{k+1}-\tau_{k}> t \right) \leq (1-\lambda_1)^{\lfloor t/\Lambda_1\rfloor},
\end{aligned}
\end{eqnarray}
where $\lambda_1\in (0,1)$ and $\Lambda_1>0$ are constants  only depending on  $n, \eta, \underline{\rho}, \alpha$, and $r_i$, $i\in\mathcal{V}$.
\end{thm}


The inequality (\ref{theo1_01}) denotes that $d_{\mathcal{V}}(t)$ will switch between $(0,\alpha]$ and $[c_\alpha^1,1]$ infinitely often, and the switching interval is a random variable depending on $n, \eta, \underline{\rho}, \alpha$, and $r_i$, $i\in\mathcal{V}$. However, the specific dependencies are very non-linear, and difficult to describe even through simulations.

\begin{rem}
From (\ref{theo1_00}), the upper limit of the maximum opinion difference $\overline{d}_{\mathcal{V}}$ has a phase transition at the point $\eta=r_{\min}/2$, providing $r_{\max}\leq \frac{n r_{\min}}{2}$. This result implies that the maximum opinion difference depends on the minimal confidence threshold among all individuals. Thus, by comparing the homogeneous and heterogeneous cases of system (\ref{m1}) with the same average confidence bound, the heterogeneity of individuals is harmful to synchronization, which may be the reason why the synchronization of opinions is hard to reach in reality, and even within that of a small group.
\end{rem}

We also provide some simulations for Theorem \ref{theorem_1}. Consider the system (\ref{m1}) with $n$ agents whose initial opinions are all set to be $0.5$. For the confidence bounds of the agents,  we set $r_1=r_{\min}=0.05$ and $r_n=r_{\max}=0.45$, and choose $\{r_i\}_{i=2}^{n-1}$ randomly and uniformly from $[0.05,0.45]$. Suppose the noises $\{\xi_i(t)\}$ are independently and uniformly distributed in $[-\eta,\eta]$. All simulations run up to $10^6$ steps. We first choose $n=20$ and $\eta=0.025=r_{\min}/2$, and the value of $d_{\mathcal{V}}(t)$ is shown in Figure \ref{Fig1n}. In this figure, it can be observed that $\overline{d}_{\mathcal{V}}=0.05=2\eta=r_{\min}$, which is consistent with (\ref{theo1_00}) and the system (\ref{m1}) reaches quasi-synchronization. Second we choose $n=20$ and $\eta=0.1>\max\{ \frac{r_{\min}}{2}, \frac{r_{\max}}{n}\}$, and the value of $d_{\mathcal{V}}(t)$ is presented in Figure \ref{Fig2n}. In this figure it can be seen that $\overline{d}_{\mathcal{V}}=1$, which is also consistent with (\ref{theo1_00}). Finally we consider a small group with $n=4$ and $\eta=0.1\in (\frac{r_{\min}}{2}, \frac{r_{\max}}{n})$, and the value of $d_{\mathcal{V}}(t)$ is provided in Figure \ref{Fig3n}. Differing from Figure \ref{Fig2n}, it seems that $\overline{d}_{\mathcal{V}}<1$, so the behavior of the system with a small number of agents is quite different from a large number of agents.

\begin{figure}[htbp]
\begin{center}
\includegraphics[height=1.8in,width=3.3in]{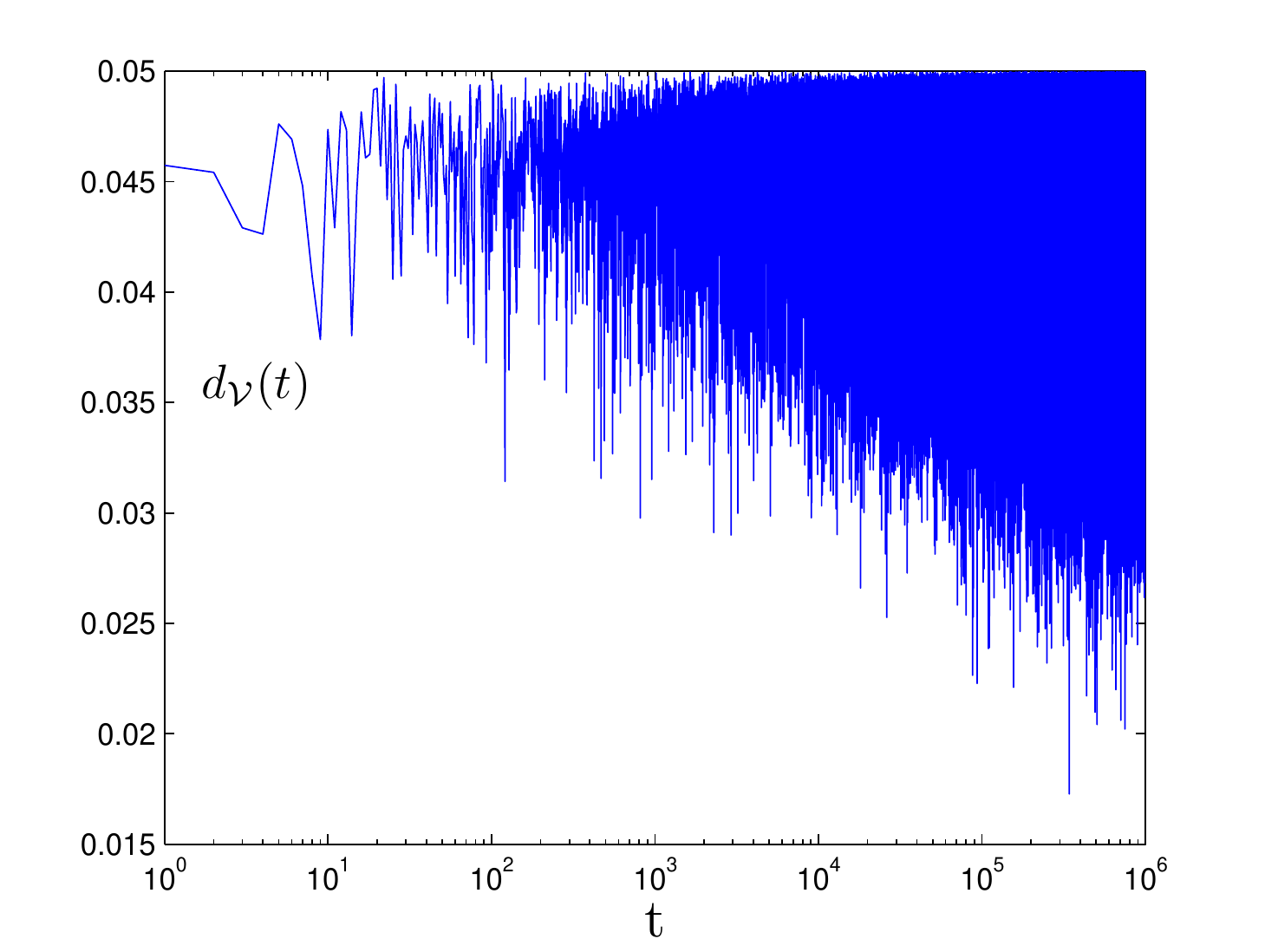}
\caption{The value of $d_{\mathcal{V}}(t)$ under system (\ref{m1}) with $n=20$ and $\eta=0.025$.} \label{Fig1n}
\includegraphics[height=1.8in,width=3.3in]{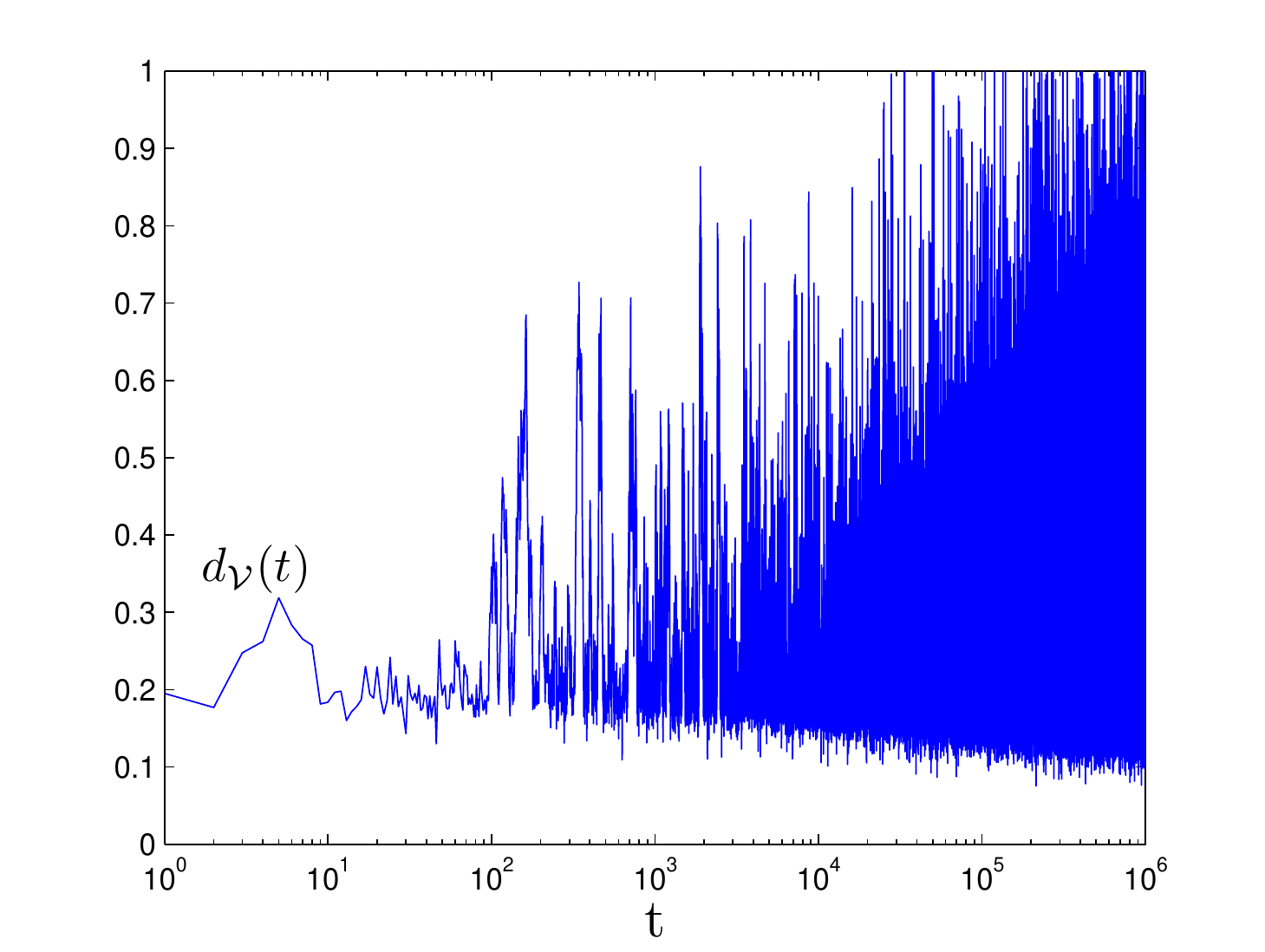}
\caption{The value of $d_{\mathcal{V}}(t)$ under system (\ref{m1}) with $n=20$ and $\eta=0.1$.} \label{Fig2n}
\includegraphics[height=1.8in,width=3.3in]{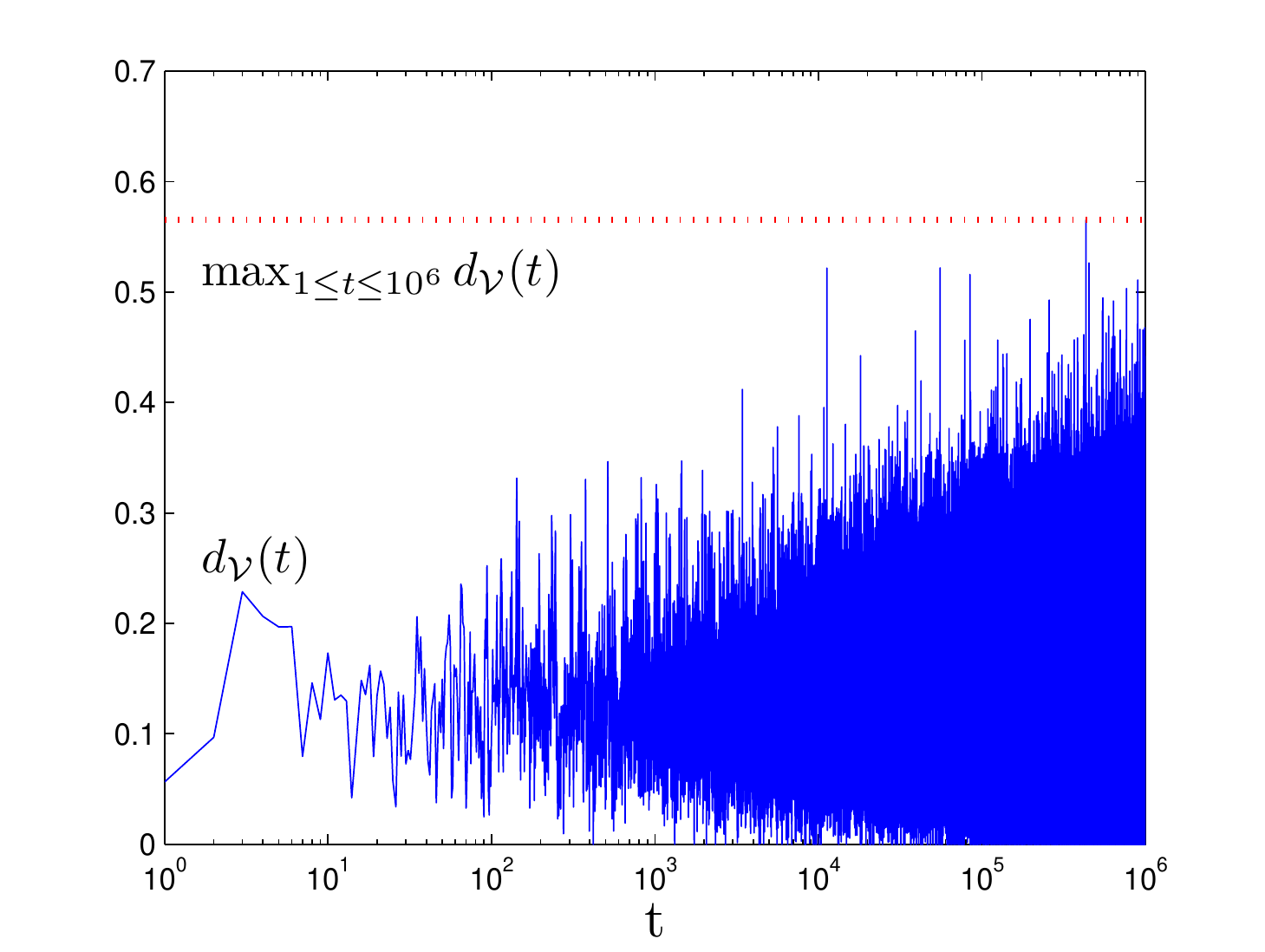}
\caption{The value of $d_{\mathcal{V}}(t)$ under system (\ref{m1}) with $n=4$ and $\eta=0.1$.} \label{Fig3n}
\end{center}
\end{figure}

By Theorem \ref{theorem_1} and Definition \ref{robconsen} we get the following corollary concerning the critical noise amplitude and convergence rate for quasi-synchronization:

\begin{cor}[Critical noise amplitude and convergence rate for quasi-synchronization of system (\ref{m1})]\label{cor_thm_1}
Consider the system (\ref{m1}) satisfying (A1) and $r_{\min}<1$.  Then, for
any initial state, the system asymptotically reaches quasi-synchronization a.s.
if and only if $\eta \leq r_{\min}/2$.

Moreover, if $\eta \leq r_{\min}/2$,  let $\tau$ be the minimal $t$ satisfying $d_{\mathcal{V}}(t)\leq 2\eta$. Then, $d_{\mathcal{V}}(t)\leq 2\eta$ for all $t\geq \tau$, and there exist constants $\lambda_2\in(0,1)$ and $\Lambda_2>0$ depending on  $n, \eta, \underline{\rho}$, and $r_i$, $i\in\mathcal{V}$ only such that
\begin{eqnarray*}\label{cor_thm_1_01}
\begin{aligned}
P\left(\tau> s \right) \leq (1-\lambda_2)^{\lfloor s/\Lambda_2\rfloor}, ~~\forall s\geq 0.
\end{aligned}
\end{eqnarray*}
\end{cor}

Next, we give our main result for system (\ref{m3}).  For any $\eta>r_{\min}/2$, let
$\underline{w}_{\eta}:=\min\{w_i:r_i<2\eta\}$  and
\begin{eqnarray*}
c_{\eta}^2=\max\Big\{2\eta,\min\Big\{\frac{n\eta}{(n-1)\underline{\omega}_{\eta}},n\eta\Big\}\Big\}.
\end{eqnarray*}
Also, for any $\alpha>0$ we define
\begin{equation*}
  c_{\alpha}^3:=\left\{
           \begin{array}{ll}
            2\eta-\alpha  & \mbox{if~} \eta\leq \frac{r_{\min}}{2} \\
            \min\{c_{\eta}^2-\alpha,1\}  &  \mbox{if~} \eta> \frac{r_{\min}}{2}
           \end{array}
         \right..
\end{equation*}

\begin{thm}[Phase transition and switching interval for the heterogeneous HK model with environment noise and global information]\label{theorem_2}
Consider system (\ref{m3}) satisfying (A1) and $r_{\min}<1$. Then
for any initial state $x(0)\in [0,1]^n$,
almost surely
\begin{equation}\label{theo2_00}
  \overline{d}_{\mathcal{V}}\left\{
           \begin{array}{ll}
            =2\eta  & \mbox{if~} \eta\leq \frac{r_{\min}}{2} \\
            \geq \min\{c_{\eta}^2,1\}  &  \mbox{if~} \eta> \frac{r_{\min}}{2}
           \end{array}
         \right..
\end{equation}
Also, for any constant $\alpha>0$, if we define $\tau_0=0$, and $\{\tau_k\}_{k\geq 1}$ as same as (\ref{theo1_00_a}) but using $c_{\alpha}^3$ instead of  $c_{\alpha}^1$,
 then for all $k\geq 0$ and $t\geq 0$ we have
\begin{eqnarray}\label{theo2_01}
\begin{aligned}
P\left(\tau_{k+1}-\tau_{k}> t \right) \leq (1-\lambda_3)^{\lfloor t/\Lambda_3\rfloor},
\end{aligned}
\end{eqnarray}
where $\lambda_3\in (0,1)$ and $\Lambda_3>0$ are constants  only depending on  $n, \eta, \underline{\rho}$,  $\alpha$, and $\omega_i, r_i$, $i\in\mathcal{V}$.
\end{thm}

\begin{rem}
Differing from system (\ref{m1}), the system (\ref{m3}) has a global average opinion $x_{\rm{ave}}(t)$, which means all agents remain connected for all time. In Theorem \ref{theorem_2}, if we let $\underline{w}_{\eta}$ tend to zero, then $\overline{d}_{\mathcal{V}}\geq \min\{ n\eta,1\}$ for the case when $\eta>r_{\min}/2$. Thus, the results of Theorems \ref{theorem_1} and \ref{theorem_2} are consistent for the case when $\eta \geq \max\{ \frac{r_{\min}}{2}, \frac{1}{n}\}$. However, Theorem \ref{theorem_1} is not a special case of Theorem \ref{theorem_2} because the former cannot be included by the latter for the case when $\max\{ \frac{r_{\min}}{2}, \frac{r_{\max}}{n}\}<\eta<\max\{ \frac{r_{\min}}{2}, \frac{1}{n}\}$.
\end{rem}

Similar to Theorem \ref{theorem_1} we also provide some simulations for Theorem \ref{theorem_2}. Consider the system (\ref{m3}) with $n=20$ agents whose initial opinions, confidence bounds, and noises have the same configurations as Figures \ref{Fig1n} and   \ref{Fig2n}. Assume $\omega_1=\cdots=\omega_n=0.1$. All simulations run up to $10^6$ steps. We first choose $\eta=0.025$ as the same as Figure \ref{Fig1n}, and the value of $d_{\mathcal{V}}(t)$ is shown in Figure \ref{Theorem37an}. This figure displays a similar behavior of $d_{\mathcal{V}}(t)$ as Figure \ref{Fig1n}. Second, we choose $\eta=0.1$ as the same as Figure \ref{Fig2n}, and the value of $d_{\mathcal{V}}(t)$ is shown in Figure \ref{Theorem37bn}. Comparing this figure to Figure \ref{Fig2n}, the system (\ref{m3}) has a smaller $\overline{d}_{\mathcal{V}}$ than the system (\ref{m1}), signifying that the global average opinion can reduce the difference of opinions.

\begin{figure}[htbp]
\begin{center}
\includegraphics[height=1.8in,width=3.3in]{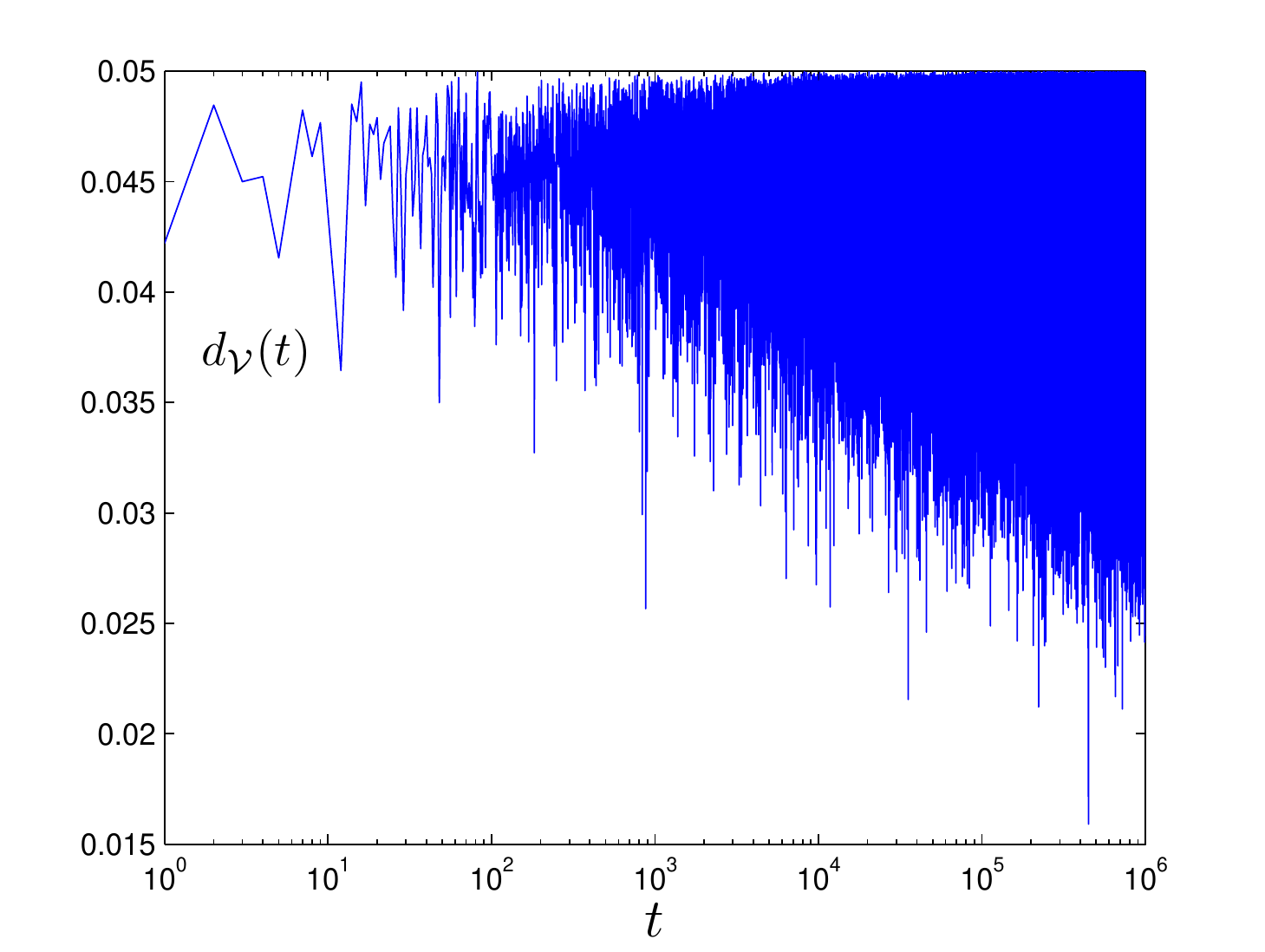}
\caption{The value of $d_{\mathcal{V}}(t)$ under system (\ref{m3}) with $\eta=0.025$.} \label{Theorem37an}
\includegraphics[height=1.8in,width=3.3in]{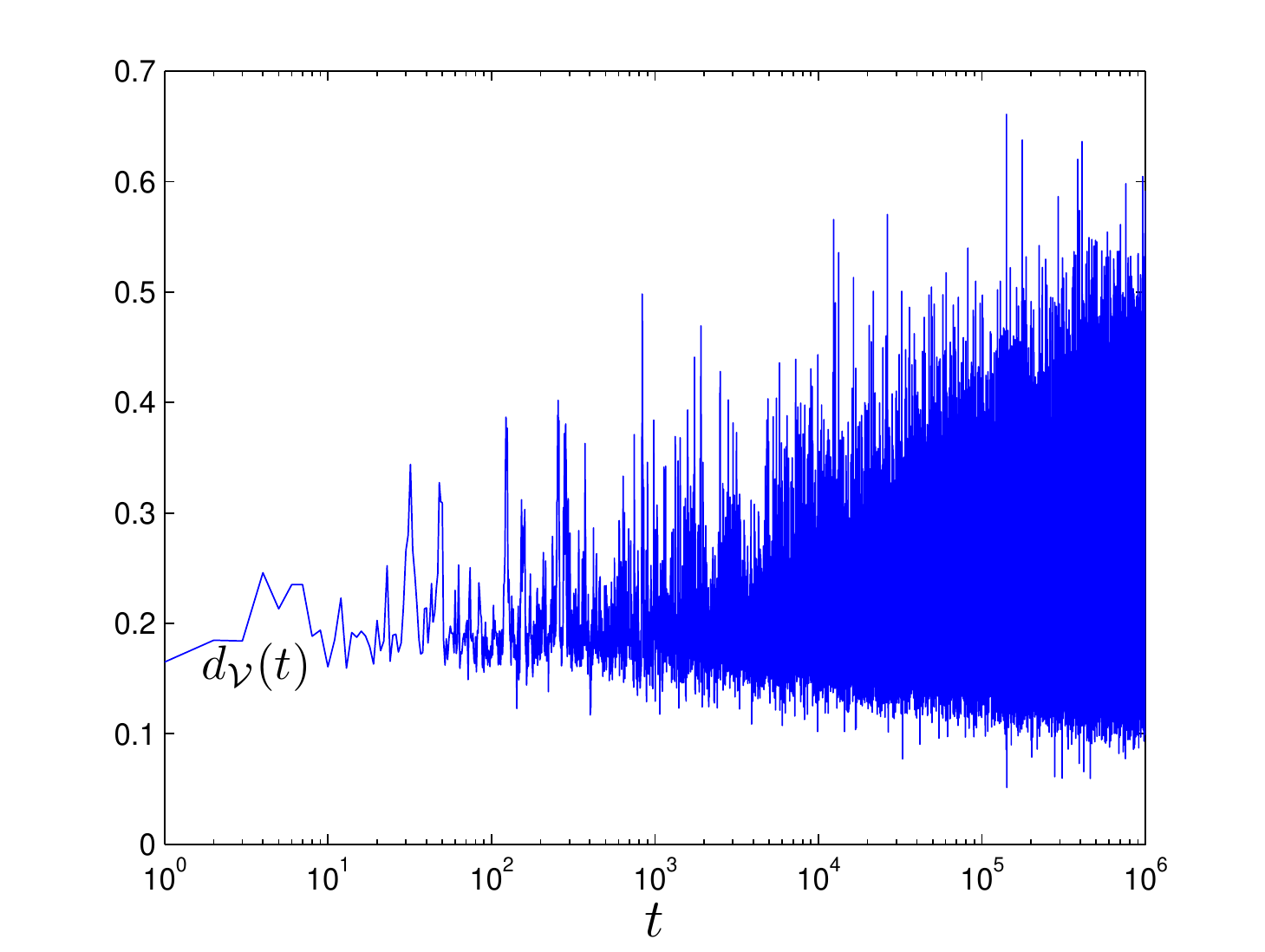}
\caption{The value of $d_{\mathcal{V}}(t)$ under system (\ref{m3}) with $\eta=0.1$.} \label{Theorem37bn}
\end{center}
\end{figure}

%

Similar to Corollary \ref{cor_thm_1} we can get the following corollary concerning the critical noise amplitude
and convergence rate for quasi-synchronization of system (\ref{m3}):

\begin{cor}[Critical noise amplitude and convergence rate for quasi-synchronization of system (\ref{m3})]\label{cor_thm_2}
Consider the system (\ref{m3}) satisfying (A1) and $r_{\min}<1$. Then, for any initial state, the system asymptotically reaches quasi-synchronization a.s.
if and only if $\eta \leq r_{\min}/2 $.

Moreover, if $\eta \leq r_{\min}/2$, let $\tau$ be the minimal $t$ satisfying $d_{\mathcal{V}}(t)\leq 2\eta$. Then, $d_{\mathcal{V}}(t)\leq 2\eta$ for all $t\geq \tau$, and there exist constants $\lambda_4\in(0,1)$ and $\Lambda_4>0$ depending on  $n, \eta, \underline{\rho}$, and $\omega_i,r_i$, $i\in\mathcal{V}$ only such that
\begin{eqnarray*}\label{cor_thm_2_01}
\begin{aligned}
P\left(\tau> s \right) \leq (1-\lambda_4)^{\lfloor s/\Lambda_4\rfloor}, ~~\forall s\geq 0.
\end{aligned}
\end{eqnarray*}
\end{cor}

\begin{rem}
Combining Corollary \ref{cor_thm_2} with Corollary \ref{cor_thm_1} shows that
 the background opinion or global information does not affect the critical noise amplitude of quasi-synchronization.
\end{rem}


Next we consider the HK model with global information and communication noise which is much more complex than
the HK model with  global information and environment noise,
 since some agents may be not affected by noise if they have no neighbor except themselves.
For $1\leq i\neq j\leq n$, define $a_i:=\frac{(n-1)(1-\omega_i)\eta}{n}$, $a_{ij}:=a_i+a_j=\frac{(n-1)(2-\omega_i-\omega_j)\eta}{n}$,
\begin{eqnarray}\label{h_00}
  &&h_{ij}\\
  &&:=\left\{
           \begin{array}{ll}
            \frac{(n-1)\eta}{n}[(1-\omega_i)^2+\frac{\omega_i(\omega_j-\omega_i)}{n}]~~\mbox{if~} r_i<a_i\\
           (1-\omega_i)\eta[\frac{1-\omega_i}{n}+\frac{n-2}{n-1}]+\frac{\omega_i(n-1)(\omega_j-\omega_i)\eta}{n^2}\\
             ~~~~~~~~~~~~~~~~~~~~~~~~~~~~~~    \mbox{if~} r_i\in[a_i,a_{ij}) \\
           \frac{(n-1)\eta}{n}[1-\omega_i+\frac{\omega_j-\omega_i}{n}] ~~ \mbox{if~} r_i\geq a_{ij}
           \end{array}
         \right.,\nonumber
\end{eqnarray}
and
\begin{eqnarray*}
 && c_{ij}:=\min\big\{h_{ij}+h_{ji},\\
  &&~~~~~~~~1-(a_i-h_{ij})I_{\{a_i>h_{ij}\}}-(a_j-h_{ji})I_{\{a_j>h_{ji}\}}\big\},
\end{eqnarray*}
where $I_{\{\cdot\}}$ denotes the indicator function.

\begin{thm}[Analytical results for the heterogeneous HK model with communication noise and global information]\label{theorem_2c}
Consider the system (\ref{m6}) satisfying (A2).
Then for any initial state, almost surely
\begin{equation*}
  \overline{d}_{\mathcal{V}}\left\{
           \begin{array}{ll}
            =\max_{i\neq j} a_{ij} ~~\mbox{if~} \eta\leq \min_{i\neq j} \frac{n r_i}{(n-1)(2-\omega_i-\omega_j)}, \\
            \geq \min\{\max_{i\neq j}\{a_{ij},c_{ij}\},1\}  ~~  \mbox{otherwise}.
           \end{array}
         \right.
\end{equation*}
\end{thm}

The result in Theorem \ref{theorem_2c} seems very complex. If we consider the large population that $n\rightarrow\infty$, and all agents has
a same belief factor $\omega^*$ in the average opinion $x_{\rm{ave}}(t)$, then by Theorem \ref{theorem_2c} we can get
\begin{equation*}
  \overline{d}_{\mathcal{V}}\left\{
           \begin{array}{ll}
            =2(1-\omega^*)\eta  &\mbox{if~} \eta\leq  \frac{r_{\min}}{2(1-\omega^*)} \\
            \geq \min\{2(1-\omega^*)\eta,1\}  & \mbox{otherwise}
           \end{array}
         \right..
\end{equation*}
In this case, we can see that $\overline{d}_{\mathcal{V}}$ still depends on the minimal confidence bound $r_{\min}$, though the phase transition is unknown. Also, it is shown that if $\eta\leq r_{\min}/2$, then  $\bar{d}_{\mathcal{V}}=2\eta(1-\omega^*)$ is dependent on $\omega^*$, whereas $\bar{d}_{\mathcal{V}}=2\eta$ in Theorem \ref{theorem_2} is independent of $\{\omega_i\}$. The reason for this difference is that the noise in system (\ref{m6})  has a product with $1-\omega_i$, whereas the noise in system (\ref{m3}) does not have a product with $1-\omega_i$.

All the above results study the upper limit $\overline{d}_{\mathcal{V}}$ of the maximum opinion difference. In fact,
for the lower limit $\underline{d}_{\mathcal{V}}$ we can get the following simple result:

\begin{thm}[The lower limit of maximum opinion difference in heterogeneous noisy HK models]\label{prop_1}
Consider systems (\ref{m1}) and (\ref{m3}) satisfying (A1), and system (\ref{m6}) satisfying (A2).
 Then for any initial state, $\underline{d}_{\mathcal{V}}=0$ a.s..
\end{thm}

Theorem \ref{prop_1} indicates that the opinions of all agents can reach almost consensus at infinite moments, however the
consensus is not a stable state and the opinions may diverge under the influence of noise.

\subsection{Proofs of Theorems \ref{theorem_1} and \ref{theorem_2}, and Corollaries \ref{cor_thm_1} and \ref{cor_thm_2}}\label{PT_1c}

We adopt the method of ``transforming the analysis of a stochastic
system into the design of control algorithms"
first proposed by \cite{Chen2017}. This method requires
the construction of some new systems to help with the analysis of the noisy HK models.
For $i\in\mathcal{V}$ and $t\geq 0$, let
\begin{eqnarray}\label{lem1_0}
\widetilde{x}_i(t)=\left\{%
\begin{array}{ll}
|\mathcal{N}_i(t)|^{-1}\sum\limits_{j\in \mathcal{N}_i(t)}x_j(t) \\
~~~~ \mbox{for protocols}~(\ref{m1})~\mbox{and}~(\ref{m4})\\
\omega_i x_{\rm{ave}}(t)+ \frac{1-\omega_i}{|\mathcal{N}_i(t)|}\sum\limits_{j\in \mathcal{N}_i(t)}x_j(t) \\
~~~~ \mbox{for protocols}~(\ref{m3})~\mbox{and}~(\ref{m6})
\end{array}%
\right..
\end{eqnarray}
From this definition the systems (\ref{m1}) and (\ref{m3}) can be rewritten as $$x_i(t+1)=\Pi_{[0,1]}(\widetilde{x}_i(t)+\xi_i(t)).$$

To analyze systems (\ref{m1}) and (\ref{m3}), we  construct two robust control systems as follows. For $i\in\mathcal{V}$ and $t\geq 0$, let $\delta_i(t)\in(0,\eta)$ be an arbitrarily given real number,
$u_i(t)\in [-\eta+\delta_i(t),\eta-\delta_i(t)]$ denotes a bounded control input, and $b_i(t)\in [-\delta_i(t),\delta_i(t)]$ denotes the
parameter uncertainty.
For protocol (\ref{m1}) we construct a control system that for all $i\in\mathcal{V}$ and $t\geq 0$,
\begin{eqnarray}\label{m1_new}
 && x_i(t+1)\\
 &&=\Pi_{[0,1]}\big( |\mathcal{N}_i(t)|^{-1}\sum\limits_{j\in \mathcal{N}_i(t)}x_j(t)+u_i(t)+b_i(t)\big).\nonumber
\end{eqnarray}
Similarly, for protocol (\ref{m3}) we construct a control system that  for all $i\in\mathcal{V}$ and $t\geq 0$,
\begin{eqnarray}\label{m3_new}
  &&x_i(t+1)=\Pi_{[0,1]}\\
  &&\Big( \omega_i x_{\rm{ave}}(t)+ \frac{1-\omega_i}{|\mathcal{N}_i(t)|}\sum\limits_{j\in \mathcal{N}_i(t)}x_j(t)+u_i(t)+b_i(t)\Big),\nonumber
\end{eqnarray}
By (\ref{lem1_0}),  the systems (\ref{m1_new}) and (\ref{m3_new}) can be rewritten as $$x_i(t+1)=\Pi_{[0,1]}(\widetilde{x}_i(t)+u_i(t)+b_i(t)).$$

 Given a set $ S\subseteq [0,1]^n$, we say $ S$ \emph{is reached at time $t$} if  $x(t)\in  S$, and \emph{is reached in the time $[t_1,t_2]$} if there exists $t'\in [t_1,t_2]$ such that  $x(t')\in  S$. Based on this, we define the robust reachability of a set as follows.

\begin{defn}\label{def_reach}
Let  $ S_1, S_2\subseteq  [0,1]^n$ be two state sets.
Under protocol (\ref{m1_new}) (or (\ref{m3_new})), $S_1$ is said to be finite-time robustly reachable from $S_2$, if, for any $x(0) \in S_2$, $S_1$ is reached at time $0$,
or there exist constants $T>0$ and $\varepsilon\in(0,\eta)$ independent of $x(0)$  such that we can find $\delta_i(t)\in[\varepsilon,\eta)$ and $u_i(t)\in [-\eta+\delta_i(t),\eta-\delta_i(t)]$, $i\in\mathcal{V}$, $0\leq t<T$ which guarantee that
$ S_1$ is reached in the time $[1,T]$ for arbitrary $b_i(t)\in [-\delta_i(t),\delta_i(t)]$,
$i\in\mathcal{V}$, $0\leq t<T$.
\end{defn}


With this definition and similar to Lemma 3.1 in \cite{Chen2017}, we get the following lemma:

\begin{lem}\label{robust}
Assume that (A1) holds.
Let $ S_1,\ldots,  S_k\subseteq  [0,1]^n$, $k\geq 1$ be state sets and assume they are finite-time robustly reachable from $[0,1]^n$
under protocol (\ref{m1_new}) (or (\ref{m3_new})). Suppose the initial opinions $x(0)$ are arbitrarily given. Then for system (\ref{m1}) (or (\ref{m3})):\\
(i)~ With probability $1$  there exists an infinite sequence $t_1<t_2<\ldots$ such that $ S_j$ is reached at time
$t_{lk+j}$ for all $j=1,\ldots,k$ and $l\geq 0$.\\
(ii)~ There exist constants $T>0$ and $c\in (0,1)$ such that
$$P\left(\tau_i-\tau_{i-1}>t\right)\leq c^{\lfloor t/T\rfloor}, \forall i,t\geq 1,$$
where $\tau_0=0$ and
$\tau_i:=\min\{s: \mbox{ there exist }\tau_{i-1}<t_1'<t_2'<\cdots<t_k'=s \mbox{ such that for all }j\in[1,k], S_j \mbox{ is reached at time }t_j'\}$ for $i\geq 1$.
\end{lem}

\begin{proof} (i)
This proof is similar to the proof of Lemma 3.1 (i) in \cite{Chen2017}. To simplify the exposition we only give a proof sketch here.
First, because $S_j (1\leq j\leq k)$ is finite-time robustly reachable under protocol (\ref{m1_new}) (or (\ref{m3_new})), so with Definition \ref{def_reach} there exist constants $T_j\geq 2$ and $\varepsilon_j\in(0,\eta)$ such that for any $t\geq 0$ and $x(t)\notin S_j$,
 we can find parameters $\delta_i(t')\in [\varepsilon_j,\eta)$ and control inputs $u_i(t')\in [-\eta+\delta_i(t'),\eta-\delta_i(t')]$, $1\leq i\leq n$, $t\leq t'\leq t+T_j-2$ with which the set $ S_j$ is reached in the time $[t+1,t+T_j-1]$ for any uncertainties $b_i(t')\in [-\delta_i(t'),\delta_i(t')]$, $1\leq i\leq n, t\leq t'\leq t+T_j-2$.
This acts on protocol (\ref{m1}) (or (\ref{m3})) indicating that for any $x(t-1)\in [0,1]^n$,
 \begin{eqnarray*}\label{rob_1}
\begin{aligned}
&P\left(\left\{\mbox{$ S_j$ is reached in $[t+1, t+T_j-1]$}\right\}| x(t-1)\right)\\
&\geq P\left(\bigcap_{t\leq t'\leq t+T_j-2}\bigcap_{1\leq i\leq n}\right.\\
&~~\left\{\xi_i(t')\in [u_i(t')-\delta_i(t'),u_i(t')+\delta_i(t')] \right\}|x(t-1)\Bigg)\\
&\geq \prod_{t'=t}^{t+T_j-2}\left[\underline{\rho} \prod_{i=1}^n (2\delta_i(t'))\right]\\
&\geq \underline{\rho}^{T_j-1} \left(2 \varepsilon_j \right)^{n (T_j-1)},
\end{aligned}
\end{eqnarray*}
where the second inequality uses (A1) and the similar discussion to (11) in \cite{Chen2017}.
Set
$$E_{j,t}:=\left\{\mbox{$ S_j$ is reached in $[t+1, t+T_j-1]$}\right\},$$
 $E_t:=\bigcap_{j=1}^k E_{j,t+\sum_{l=1}^{j-1} T_l}$, and $T:=T_1+T_2+\ldots+T_k$.
Similar to (13) in \cite{Chen2017} we get
\begin{eqnarray*}\label{rob_3}
&&P\Big(\bigcap_{m=M}^{\infty}E_{mT}^c\Big)=0 \mbox{~for all~} M>0,
\end{eqnarray*}
which indicates that with probability $1$ there exists an infinite sequence $m_1<m_2<\ldots$ such that $E_{m_l T}$ occurs for all $l\geq 1$.
By the definition of $E_t$, for each $l\geq 0$ we can find a time sequence $t_{lk+j}\in [m_l T+\sum_{p=1}^{j-1} T_p, m_l T+\sum_{p=1}^{j} T_p-1]$, $1\leq j\leq k$ such that $S_j$ is reached at time $t_{lk+j}$.

(ii) Same as the proof of Lemma 3.1 (ii) in \cite{Chen2017}.
\end{proof}


Lemma \ref{robust} builds a connection between the noisy HK system (\ref{m1}) (or (\ref{m3})) and the HK-control system (\ref{m1_new}) (or (\ref{m3_new})). According to Lemma \ref{robust}, to prove a set $S$ is reached a.s. under a noisy HK system, we only need to design control algorithms for the corresponding HK-control system such that the set $S$ is robustly reached in finite time. Before the proof of Theorem \ref{theorem_1} we introduce
some useful notions and lemmas.

For any constants $z\in [0,1]$ and $\alpha>0$, define the set
\begin{equation}\label{set_less}
S_{z,\alpha}:=\Big\{(x_1,\ldots,x_n)\in[0,1]^n: \max_{i\in\mathcal{V}}|x_i-z|<\alpha\Big\}.
\end{equation}

\begin{lem}\label{lem1}
For any constants $z'\in [0,1]$ and $\alpha'>0$, $S_{z',\alpha'}$ is finite-time robustly reachable from $[0,1]^n$ under protocols
(\ref{m1_new}) and (\ref{m3_new}).
\end{lem}

For any constant $\beta\in[0,1]$, denote the set
\begin{equation}\label{set_big}
E_{\beta}:=\Big\{(x_1,\ldots,x_n)\in[0,1]^n: \max_{i,j\in\mathcal{V}}|x_i-x_j| \geq \beta\Big\}.
\end{equation}

\begin{lem}\label{lem2}
If $r_{\min}<1$ and $\eta> \max\{ \frac{r_{\min}}{2}, \frac{r_{\max}}{n}\}$,
 $E_1$ is finite-time robustly reachable from $[0,1]^n$ under protocol
(\ref{m1_new}).
\end{lem}

\begin{lem}\label{lem3}
For any constant $\varepsilon\in(0,\eta)$, let $c_{\varepsilon}=\min\{2\eta-2\varepsilon,1\}$.
Then $E_{c_{\varepsilon}}$ is finite-time robustly reachable from $[0,1]^n$ under protocols
(\ref{m1_new}) and (\ref{m3_new}).
\end{lem}

\begin{lem}\label{lem4}
Suppose $r_{\min}<1$ and $\eta>\frac{r_{\min}}{2}$. For any constant $\varepsilon>0$, let
$$c_{\varepsilon}:=\min\Big\{\frac{n\eta}{(n-1)\underline{\omega}_{\eta}}-\varepsilon,n\eta-\varepsilon,1\Big\},$$
then
 $E_{c_\varepsilon}$ is finite-time robustly reachable from $[0,1]^n$ under protocol
(\ref{m3_new}).
\end{lem}

The proofs of Lemmas \ref{lem1}-\ref{lem4} are postponed to Appendix \ref{Proof_Lemmas_a}.

\begin{proof}[Proof of Theorem \ref{theorem_1}]
 By Lemma \ref{lem1},  $S_{\frac{1}{2},\frac{r_{\min}}{2}}$ is finite-time robustly reachable from $[0,1]^n$ under protocol
(\ref{m1_new}), then by   Lemma \ref{robust}, under system (\ref{m1}) a.s. there exists
 a finite time $t_1$ such that  $S_{\frac{1}{2},\frac{r_{\min}}{2}}$ is reached at time $t_1$, which indicates $d_{\mathcal{V}}(t_1)<r_{\min}$.
From this and (\ref{lem1_0}) we get $\widetilde{x}_1(t_1)=\widetilde{x}_2(t_1)=\cdots=\widetilde{x}_n(t_1)$, and so $d_{\mathcal{V}}(t_1+1)\leq 2\eta$.
Thus, if $\eta\leq \frac{r_{\min}}{2}$, we have $d_{\mathcal{V}}(t_1+1)\leq r_{\min}$. Repeating this process we get
\begin{eqnarray}\label{pf_thm1_1}
d_{\mathcal{V}}(t)\leq 2\eta,~~~~\forall t>t_1
\end{eqnarray}
for $\eta\leq \frac{r_{\min}}{2}$.

Also, for any small constant $\varepsilon>0$, by Lemmas \ref{robust} and \ref{lem3} we get a.s.
 $E_{2\eta-\varepsilon}$ is reached in finite time under protocol (\ref{m1}), so making $\varepsilon$ tend to zero we have
\begin{eqnarray}\label{pf_thm1_2}
\overline{d}_{\mathcal{V}}\geq 2\eta ~~~~\mbox{a.s.}
 \end{eqnarray}

For the case when $\eta\leq \frac{r_{\min}}{2}$, by (\ref{pf_thm1_1}) and (\ref{pf_thm1_2}) we have $\overline{d}_{\mathcal{V}}=2\eta$ a.s. For the case when $\eta> \max\{ \frac{r_{\min}}{2}, \frac{r_{\max}}{n}\}$, by Lemmas \ref{robust} and \ref{lem2} we can get $\overline{d}_{\mathcal{V}}= 1$ a.s.

It remains to prove (\ref{theo1_01}).  Using Lemma \ref{lem1} again, we have $S_{\frac{1}{2},\frac{\alpha}{2}}$ is finite-time robustly reachable from $[0,1]^n$ under protocol (\ref{m1_new}). Also, by  Lemmas \ref{lem2} and \ref{lem3} we have $E_{c_{\alpha}^1}$ is finite-time robustly reachable from $[0,1]^n$ under protocol (\ref{m1_new}). Combining these with Lemma \ref{robust} (ii)  yields (\ref{theo1_01}).
\end{proof}

\begin{proof}[Proof of Corollary \ref{cor_thm_1}]
If $\eta\leq \frac{r_{\min}}{2}$, by (\ref{theo1_00}) we have $\overline{d}_{\mathcal{V}}=2\eta\leq r_{\min}$ a.s.,
which means the system asymptotically reaches quasi-synchronization a.s. by Definition \ref{robconsen}.
If $\eta>\frac{r_{\min}}{2}$, by (\ref{theo1_00}) we get a.s. $\overline{d}_{\mathcal{V}}\geq 2\eta >r_{\min}$ a.s.  Thus, the system asymptotically reaches quasi-synchronization a.s.
if and only if $\eta \leq r_{\min}/2$.

For the case of $\eta\leq \frac{r_{\min}}{2}$,
using (\ref{theo1_01}) with $\alpha=2\eta$ there are constants $\lambda_2\in(0,1)$ and $\Lambda_2>0$ such that
\begin{eqnarray*}
\begin{aligned}
P\left(\tau> s \right) \leq (1-\lambda_2)^{\lfloor s/\Lambda_2\rfloor}, ~~\forall s\geq 0.
\end{aligned}
\end{eqnarray*}
Also, by (\ref{pf_thm1_1}) but using $\tau$ instead of $t_1$ we have $d_{\mathcal{V}}(t)\leq 2\eta$ for all $t>\tau$.
\end{proof}

\begin{proof}[Proof of Theorem \ref{theorem_2}]
If $\eta\leq \frac{r_{\min}}{2}$, with the same discussion as the proof of Theorem \ref{theorem_1} we can get $\overline{d}_{\mathcal{V}}= 2\eta$ a.s..
If $\eta> \frac{r_{\min}}{2}$, for any small constant $\varepsilon>0$, by Lemmas \ref{robust}, \ref{lem3}, and \ref{lem4} we get a.s.
 $E_{\min\{c_{\eta}^2-\varepsilon,1\}}$ is reached in finite time under protocol (\ref{m3}), so making $\varepsilon$ tend to zero we have $\overline{d}_{\mathcal{V}}\geq \min\{c_{\eta}^2,1\}$ a.s..

Similar to the proof of (\ref{theo1_01}) we can get (\ref{theo2_01}).
\end{proof}

\begin{proof}[Proof of Corollary \ref{cor_thm_2}]
i) Immediately from (\ref{theo2_00}) and Definition \ref{robconsen}. \\
ii) Similar to the proof of Corollary \ref{cor_thm_1}.
\end{proof}

\subsection{Proofs of Theorems \ref{theorem_2c} and \ref{prop_1}}\label{PT_2c}
As a convenience, this subsection also provides some preparations to analyze the system (\ref{m4}) besides the system (\ref{m6}).
Similar to Subsection \ref{PT_1c}, we construct two robust control systems, which can transform the analysis of systems (\ref{m4}) and (\ref{m6})
to the design of control algorithms. For $t\geq 0$, $i\in\mathcal{V}$, and  $j\in\mathcal{N}_i(t)\backslash\{i\}$, let $\delta_i(t)\in(0,\eta)$ be an arbitrarily given real number,
$u_{ji}(t)\in [-\eta+\delta_i(t),\eta-\delta_i(t)]$ denotes a bounded control input, and $b_{ji}(t)\in [-\delta_i(t),\delta_i(t)]$ denotes the
parameter uncertainty.

Similar to (\ref{m1_new}) and (\ref{m3_new}),
for systems (\ref{m4}) and (\ref{m6}) we construct the control systems
\begin{eqnarray}\label{m4_new}
&&x_i(t+1)\nonumber\\
&&=\Pi_{[0,1]}\Big( |\mathcal{N}_i(t)|^{-1}\Big[x_i(t)\nonumber\\
&&~~~~+\sum\limits_{j\in \mathcal{N}_i(t)\backslash\{i\}}(x_j(t)+u_{ji}(t)+b_{ji}(t))\Big]\Big)\nonumber\\
&&=\Pi_{[0,1]}\Big(\widetilde{x}_i(t)+ |\mathcal{N}_i(t)|^{-1}\sum\limits_{j\in \mathcal{N}_i(t)\backslash\{i\}}(u_{ji}(t)+b_{ji}(t))\Big),\nonumber\\
&&~~~~~\forall i\in\mathcal{V},t\geq 0,
\end{eqnarray}
and
\begin{eqnarray}\label{m6_new}
 && x_i(t+1)\nonumber\\
 &&= \Pi_{[0,1]}\Big(  \omega_i x_{\rm{ave}}(t)+ \frac{1-\omega_i}{|\mathcal{N}_i(t)|}\Big[x_i(t)\nonumber\\
 &&~~~~+\sum\limits_{j\in \mathcal{N}_i(t)\backslash\{i\}}(x_j(t)+u_{ji}(t)+b_{ji}(t))\Big]\Big)\nonumber\\
  &&=\Pi_{[0,1]}\Big(\widetilde{x}_i(t)+\frac{1-\omega_i}{|\mathcal{N}_i(t)|}\sum\limits_{j\in \mathcal{N}_i(t)\backslash\{i\}}(u_{ji}(t)+b_{ji}(t))\Big),\nonumber\\
 && ~~~~~\forall i\in\mathcal{V},t\geq 0
\end{eqnarray}
respectively, where the last lines of (\ref{m4_new}) and (\ref{m6_new}) use (\ref{lem1_0}).

Similar to Definition \ref{def_reach},  we define the robust reachability  for systems (\ref{m4_new}) and (\ref{m6_new}) as follows:

\begin{defn}\label{def_reach2}
Let  $ S_1, S_2\subseteq  [0,1]^n$ be two state sets.
Under protocol (\ref{m4_new}) (or (\ref{m6_new})), $S_1$ is said to be finite-time robustly reachable from $S_2$ if: For any $x(0) \in S_2$, $S_1$ is reached at time $0$,
or there exist constants $T>0$ and $\varepsilon\in(0,\eta)$ such that we can find $\delta_i(t)\in[\varepsilon,\eta)$ and $u_{ji}(t)\in [-\eta+\delta_i(t),\eta-\delta_i(t)]$, $0\leq t<T$, $i\in\mathcal{V}$, $j\in\mathcal{N}_i(t)\backslash\{i\}$,  which guarantees that
$ S_1$ is reached in the time $[1,T]$ for arbitrary $b_{ji}(t)\in [-\delta_i(t),\delta_i(t)]$,  $0\leq t<T$,
$i\in\mathcal{V}$, $j\in\mathcal{N}_i(t)\backslash\{i\}$.
\end{defn}

Similar to Lemma \ref{robust} we get the following lemma:
\begin{lem}\label{robust2}
Assume (A2) holds.
Let $ S_1,\ldots,  S_k\subseteq  [0,1]^n$, $k\geq 1$ be state sets and assume they are finite-time robustly reachable from $[0,1]^n$
under protocol (\ref{m4_new}) (or (\ref{m6_new})). Suppose the initial opinions $x(0)$ are arbitrarily given. Then for system (\ref{m4}) (or (\ref{m6})):\\
(i)~ With probability $1$  there exists an infinite sequence $t_1<t_2<\ldots$ such that $ S_j$ is reached at time
$t_{lk+j}$ for all $j=1,\ldots,k$ and $l\geq 0$.\\
(ii)~ There exist constants $T>0$ and $c\in (0,1)$ such that
$$P\left(\tau_i-\tau_{i-1}>t\right)\leq c^{\lfloor t/T\rfloor}, \forall i,t\geq 1,$$
where $\tau_0=0$ and
$\tau_i:=\min\{s: \mbox{ there exist }\tau_{i-1}<t_1'<t_2'<\cdots<t_k'=s \mbox{ such that for all }j\in[1,k], S_j \mbox{ is reached at time }t_j'\}$ for $i\geq 1$.
\end{lem}

Similar to Lemmas \ref{lem1}, \ref{lem3} and \ref{lem4}  we can get
\begin{lem}\label{lem1_a}
For any constants $\bar{z}\in [0,1]$ and $\bar{\alpha}>0$, $S_{\bar{z},\bar{\alpha}}$ is finite-time robustly reachable from $[0,1]^n$ under protocol
(\ref{m6_new}).
\end{lem}

\begin{lem}\label{lem3_a}
For any constant $\varepsilon\in(0,\eta)$, let $c_{\varepsilon}:=\min\{\max_{i\neq j}a_{ij}-2\varepsilon,1\}.$
Then $E_{c_{\varepsilon}}$ is finite-time robustly reachable from $[0,1]^n$ under protocol
(\ref{m6_new}).
\end{lem}

\begin{lem}\label{lem4_a}
Suppose that $n\geq 3$. For any constant $\varepsilon>0$, let
$c_{\varepsilon}:=\max_{i\neq j}c_{ij}-\varepsilon$, then
 $E_{c_\varepsilon}$ is finite-time robustly reachable from $[0,1]^n$ under protocol
(\ref{m6_new}).
\end{lem}

The proof of Lemma \ref{lem1_a} is similar to the proof of Lemma \ref{lem1}, while the proofs of Lemmas \ref{lem3_a} and \ref{lem4_a}
are postponed to Appendix \ref{Proof_Lemmas_b}.

\begin{proof}[Proof of Theorem \ref{theorem_2c}]
For any constant $\varepsilon\geq 0$, let $c_{\varepsilon}^*=\min\{\max_{i\neq j}\{a_{ij},c_{ij}\},1\}-\varepsilon$.
Combining Lemma \ref{robust2} with Lemmas \ref{lem3_a} and \ref{lem4_a} we have a.s. $E_{c_{\varepsilon}^*}$ is reached in finite time under protocol
(\ref{m6}),  so making $\varepsilon$ tend to zero we have
\begin{eqnarray}\label{theo_2c_1}
\overline{d}_{\mathcal{V}}\geq c_0^* ~~~~~~a.s..
\end{eqnarray}

For the case that  $\eta\leq \min_{i\neq j} \frac{n r_i}{(n-1)(2-\omega_i-\omega_j)}$, this implies that $r_i\geq a_{ij}$ for all $i\neq j$, and then
\begin{eqnarray*}\label{theo_2c_2}
c_{ij}\leq h_{ij}+h_{ji}=a_{ij}\leq r_i\leq 1,~~~~\forall i\neq j.
\end{eqnarray*}
By this with (\ref{theo_2c_1}) we have
\begin{eqnarray}\label{theo_2c_3}
\overline{d}_{\mathcal{V}}\geq \max_{i\neq j} a_{ij}~~~~a.s..
\end{eqnarray}
Also, by Lemmas \ref{robust2} and \ref{lem1_a} we get
a.s. there exists a finite time $t_1$ such that $|x_i(t_1)-x_j(t_1)|\leq a_{ij}\leq \min\{r_i,r_j\}$ for any $i\neq j$,
which indicates $\widetilde{x}_1(t_1)=\cdots=\widetilde{x}_n(t_1)$. Here we recall that $\widetilde{x}_i(t)$ is defined by
(\ref{lem1_0}). In addition,
\begin{eqnarray*}\label{theo_2c_4}
x_i(t_1+1)=\Pi_{[0,1]}\Big(\widetilde{x}_i(t_1)+\frac{1-\omega_i}{n} \sum_{j\neq i} \zeta_{ji}(t_1)\Big),
\end{eqnarray*}
so
\begin{eqnarray*}\label{theo_2c_5}
\begin{aligned}
&|x_i(t_1+1)-x_j(t_1+1)|\\
&\leq \frac{1-\omega_i}{n} \sum_{k\neq i}|\zeta_{ki}(t_1)|+\frac{1-\omega_j}{n} \sum_{k\neq j}|\zeta_{kj}(t_1)|\\
&\leq \frac{(2-\omega_i-\omega_j)(n-1)\eta}{n}\\
&=a_{ij} \leq \min\{r_i,r_j\}.
\end{aligned}
\end{eqnarray*}
 Repeating this process we get $|x_i(t)-x_j(t)|\leq a_{ij}$ for all
$t>t_1$, which implies that  $\overline{d}_{\mathcal{V}}\leq \max_{i\neq j} a_{ij}$ a.s..
Combining this with (\ref{theo_2c_3}) we get $\overline{d}_{\mathcal{V}}=\max_{i\neq j} a_{ij}$ a.s..
\end{proof}

\begin{proof}[Proof of Theorem \ref{prop_1}]
For any $\alpha>0$, by Lemmas \ref{robust} and \ref{lem1} the set $S_{1/2,\alpha}$ is reached a.s. in finite time under systems (\ref{m1}) and (\ref{m3}),
so  we can get $\underline{d}_{\mathcal{V}}=0$ a.s. by  making $\alpha$ tend to zero.

Similarly, by  Lemmas \ref{robust2} and \ref{lem1_a} the set $S_{1/2,\alpha}$ is reached a.s. in finite time under systems (\ref{m6}),
so we can get $\underline{d}_{\mathcal{V}}=0$ a.s. by  making $\alpha$ tend to zero.
\end{proof}

\renewcommand{\thesection}{\Roman{section}}
\section{Increasing Confidence Thresholds May Harm Synchronization under System (\ref{m4})}\label{Analysis_sec_3}
\renewcommand{\thesection}{\arabic{section}}
From the study of Section \ref{Analysis_sec_1}, we see that small noise will lead to quasi-synchronization for the HK model with environment noise. For the HK model with communication noise, this result still holds in the homogeneous case, though it may be not true in the heterogeneous case. Interestingly, we will show that the quasi-synchronization may be broken if the confidence thresholds of constituent agents are increased.

  We only give the analytic result for the homogenous case of the system (\ref{m4}).
If $r_1=r_2=\cdots=r_n<\frac{1}{n-1}$, the values of $\overline{d}_{\mathcal{V}}$ and $\underline{d}_{\mathcal{V}}$ depend on the initial opinion. For example, if $x_i(0)=\frac{i-1}{n-1}$ for $1\leq i\leq n$, then all agents are isolated and will remain unchanged, which indicates $\overline{d}_{\mathcal{V}}=\underline{d}_{\mathcal{V}}=1$ for any $\eta$. However, if  $|x_i(0)-x_j(0)|\leq r_1$ for all $ i\neq j$, it can be obtained that $|x_i(t)-x_j(t)|\leq r_1$ for all $t\geq 1$ when $\eta\leq \frac{nr_1}{2(n-1)}$, which indicates $\underline{d}_{\mathcal{V}}\leq \overline{d}_{\mathcal{V}}\leq r_1$. Thus, to avoid dependence on the initial opinion, this paper only considers the case when $r_1=\cdots=r_n\geq \frac{1}{n-1}$.

   First we need introduce some lemmas as follows:
\begin{lem}\label{lem1c}
Consider the  protocol
(\ref{m4_new}) satisfying $r_1=\cdots=r_n\geq \frac{1}{n-1}$.
Then for any constants $z^*\in [0,1]$ and $\alpha^*>0$, $S_{z^*,\alpha^*}$ is finite-time robustly reachable from $[0,1]^n$,
where $S_{z^*,\alpha^*}$ is the state set defined by (\ref{set_less}).
\end{lem}

\begin{lem}\label{lem2c}
Consider the protocol (\ref{m4_new}) satisfying $r_1=\cdots=r_n\in[\frac{1}{n-1},1)$. Let $c_{\eta}:=\frac{2\eta(n-1)}{n}$.
For any $\varepsilon>0$, if $\eta\leq \frac{nr_1}{2(n-1)}$
then the set
$$E_{\varepsilon}':=\Big\{(x_1,\ldots,x_n): c_{\eta}-\varepsilon \leq \max|x_i-x_j|\leq c_{\eta}\Big\}$$
is finite-time robustly reachable from $[0,1]^n$.
\end{lem}

\begin{lem}\label{lem3c}
Consider the protocol (\ref{m4_new}) satisfying $n\geq 3$ and $r_1=\cdots=r_n\in[\frac{1}{n-1},1)$. If $\eta>\frac{nr_1}{2(n-1)}$ then
$E_1$ is finite-time robustly reachable from $[0,1]^n$,
where $E_1$  is the state set defined by (\ref{set_big}).
\end{lem}

The proofs of Lemmas \ref{lem1c}, \ref{lem2c}, and \ref{lem3c} are postponed to Appendix \ref{Proof_Lemmas_c}.

\begin{thm}\label{theorem_1c}
Consider the system (\ref{m4}) satisfying (A2), $n\geq 3$ and $r_1=\cdots=r_n\in [\frac{1}{n-1},1)$.
Then for any initial opinions $x(0)\in [0,1]^n$,
a.s. $\underline{d}_{\mathcal{V}}=0$, and
\begin{equation*}\label{theo1c_00}
  \overline{d}_{\mathcal{V}}=\left\{
           \begin{array}{ll}
           \frac{2\eta(n-1)}{n}  & \mbox{if~}  \eta\leq \frac{nr_1}{2(n-1)} \\
           1 &  \mbox{otherwise}
           \end{array}
         \right..
\end{equation*}
Also, for any $\alpha>0$ we set
\begin{equation*}
  c_{\alpha}:=\left\{
           \begin{array}{ll}
            \frac{2\eta(n-1)}{n}-\alpha  & \mbox{if~} \eta\leq \frac{nr_1}{2(n-1)} \\
            1  &  \mbox{otherwise}
           \end{array}
         \right.,
\end{equation*}
and define $\tau_0=0$, and $\tau_k$ to be the stopping time as
 \begin{eqnarray*}\label{theo1c_00_a}
\tau_k=\left\{%
\begin{array}{ll}
\min\{s>\tau_{k-1}:d_{\mathcal{V}}(s)\leq \alpha\} & \mbox{if $k$ is odd},\\
\min\{s>\tau_{k-1}:d_{\mathcal{V}}(s)\geq c_\alpha\} & \mbox{if $k$ is  even},
\end{array}%
\right.
\end{eqnarray*}
 then for all $j\geq 0$ and $t\geq 0$,
\begin{eqnarray}\label{theo1c_01}
\begin{aligned}
P\left(\tau_{2j+2}-\tau_{2j}> t \right) \leq (1-\lambda_5)^{\lfloor t/\Lambda_5\rfloor},
\end{aligned}
\end{eqnarray}
where $\lambda_5\in (0,1)$ and $\Lambda_5>0$ are constants  only depending on  $n, \eta, \underline{\rho}, \alpha$ and $r_1$.
\end{thm}
\begin{proof}
First by Lemmas \ref{lem1c} and \ref{robust2} we can get a.s. $\underline{d}_{\mathcal{V}}=0$ when we let the value of $\alpha$ in Lemma \ref{lem1c}  tend to $0$.

Next we consider the value of $\overline{d}_{\mathcal{V}}$.
For the case that $\eta\leq \frac{nr_1}{2(n-1)}$, since  $\underline{d}_{\mathcal{V}}=0$ a.s. there exists a finite time $t_1$
such that $\max_{i,j}|x_i(t_1)-x_j(t_1)|\leq \frac{n-1}{n}2\eta \leq r_1$.
By (\ref{m4}) and (\ref{lem1_0}) we have $\widetilde{x}_1(t_1)=\cdots=\widetilde{x}_n(t_1)$, and
\begin{eqnarray*}
x_i(t_1+1)\in \Big[\widetilde{x}_i(t_1)-\frac{n-1}{n}\eta, \widetilde{x}_i(t_1)+\frac{n-1}{n}\eta\Big],
\end{eqnarray*}
which indicates
\begin{eqnarray*}
|x_j(t+1)-x_i(t+1)| \leq \frac{n-1}{n}2\eta\leq r_1.
\end{eqnarray*}
Repeating the above process we get that for any $t\geq t_1$,
\begin{eqnarray}\label{theo1c_06}
\begin{aligned}
\max_{i,j} |x_{i}(t)-x_{j}(t)| \leq \frac{n-1}{n}2\eta.
\end{aligned}
\end{eqnarray}
Combining (\ref{theo1c_06}) with Lemmas \ref{lem2c} and \ref{robust2} we get $\overline{d}_{\mathcal{V}}=\frac{2\eta(n-1)}{n}$ a.s. when we let the value of $\varepsilon$ in Lemma \ref{lem2c}  tend to $0$.

If $\eta>\frac{nr_1}{2(n-1)}$, by Lemmas \ref{lem3c} and \ref{robust2} we have a.s. $\overline{d}_{\mathcal{V}}=1$.

Finally, combining Lemma (\ref{robust2}) with Lemmas \ref{lem1c}, \ref{lem2c} and \ref{lem3c} yields (\ref{theo1c_01}).
\end{proof}

\begin{rem}
In Theorem \ref{theorem_2c}, if $\{\omega_i\}_{1\leq i\leq n}$ all tend to $0^+$ then almost surely
\begin{equation}\label{rem_label1}
  \overline{d}_{\mathcal{V}}\left\{
           \begin{array}{ll}
            =\frac{2(n-1)\eta}{n} &\mbox{if~} \eta\leq \min_{i} \frac{n r_i}{2(n-1)}, \\
            \geq \min\{\max_{i\neq j}\{\frac{2(n-1)\eta}{n},c_{ij}^*\},1\}  &  \mbox{otherwise},
           \end{array}
         \right.
\end{equation}
where $c_{ij}^*$ is the limit value of $c_{ij}$ as $\{\omega_i\}_{1\leq i\leq n}$ all tend to $0^+$. Theorem \ref{theorem_1c} is consistent
with (\ref{rem_label1}) for the case when $\eta\leq \min_{i} \frac{n r_i}{2(n-1)}$, but is stronger than (\ref{rem_label1}) for the case when $\eta>\min_{i} \frac{n r_i}{2(n-1)}$. Thus, Theorem \ref{theorem_1c} is not a special case of Theorem \ref{theorem_2c}.
\end{rem}

For the system (\ref{m4}), we only consider the homogeneous case since the heterogeneous case is quite difficult to analyze in detail. When considering the heterogeneous system (\ref{m4}), the finite-time robust reachability (Lemma \ref{lem1c}) may not hold under some configurations. Then, in contrast to the system (\ref{m1}), in system (\ref{m4}), small noise may not lead to quasi-synchronization. More interestingly, if we increase the confidence thresholds of constituent agents, the quasi-synchronization may be broken. Also, the value of  $\overline{d}_{\mathcal{V}}$ as a function of $\eta$ may exhibit a phase transition at some critical points. However, raising $\eta$ may promote synchronization, which is different from the systems (\ref{m1}) and (\ref{m3}). To illustrate these phenomena we give an example as follows:

\begin{example}
Assume that system (\ref{m4}) contains $4$ agents, and the communication noises are independently and uniformly distributed in $[-\eta,\eta]$. If $r_1=r_2=r_3=r_4=\frac{1}{3}$, for any $\eta\in (0,2/9)$
and any initial opinions, by Theorem \ref{theorem_1c} the system will reach quasi-synchronization a.s.. The evolution of opinions
is shown in Figure \ref{EX46an} with $\eta=0.1$.
 However, if we
increase the values of $r_2$ and $r_3$ to $1$, and
suppose $x(0)=(0,1/2,1/2,1)$, then for any $\eta\in(0,1/9)$,
it can be shown that $x_1(t)=0$, $x_4(t)=1$ for all $t\geq 1$, while $x_2(t)$ and $x_3(t)$ fluctuate between
$(\frac{1}{2}-\frac{3}{2}\eta,\frac{1}{2}+\frac{3}{2}\eta)$.  The evolution of opinions
is shown in Figure \ref{EX46bn} with $\eta=0.1$.
\end{example}

\begin{figure}[htbp]
\begin{center}
\includegraphics[height=2in,width=2.8in]{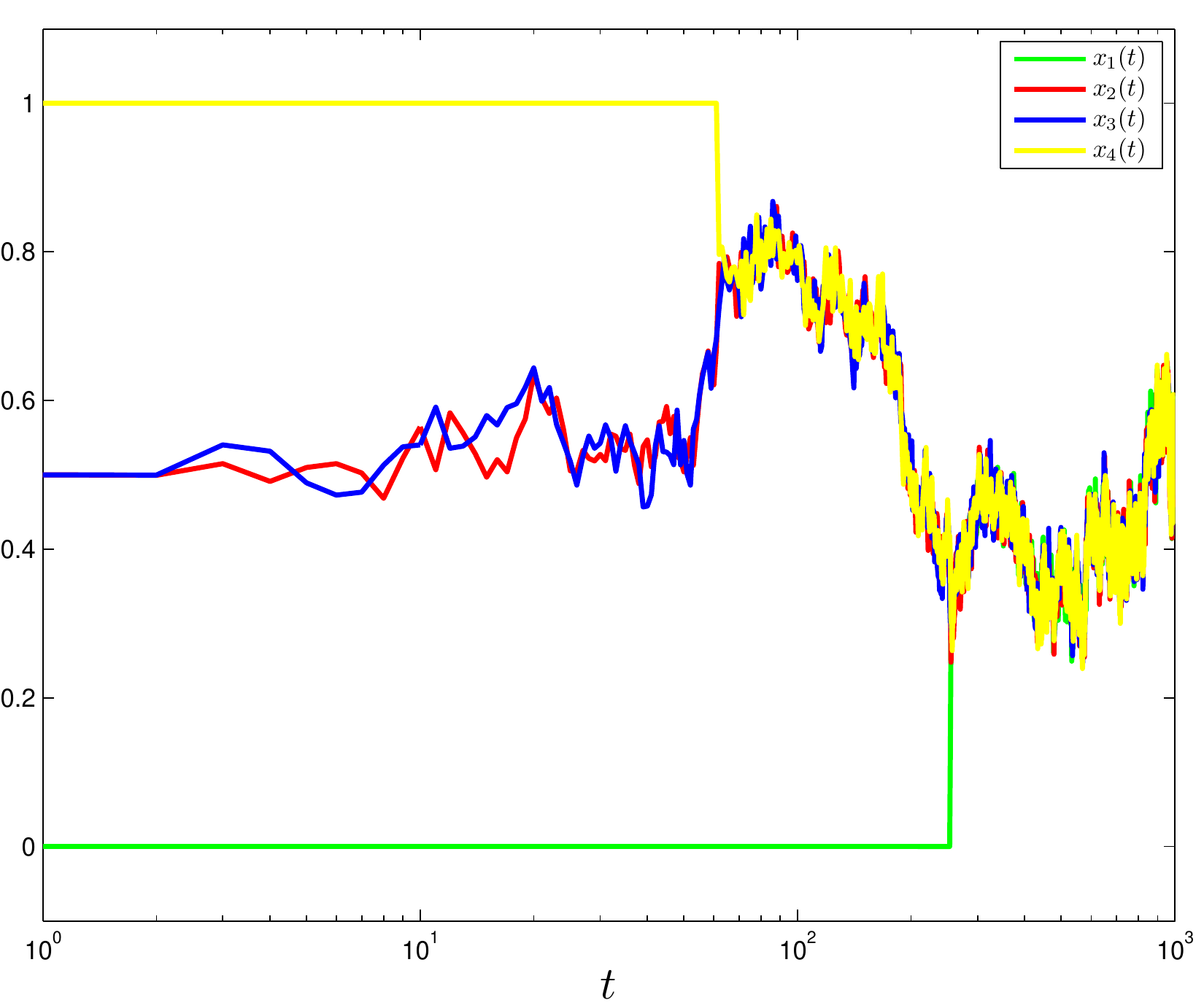}
\caption{The evolution of system (\ref{m4}) with $n=4$,
$x(0)=(0,1/2,1/2,1)$, $(r_1,\ldots,r_4)=(1/3,1/3,1/3,1/3)$, and $\eta=0.1$.} \label{EX46an}
\includegraphics[height=2in,width=2.8in]{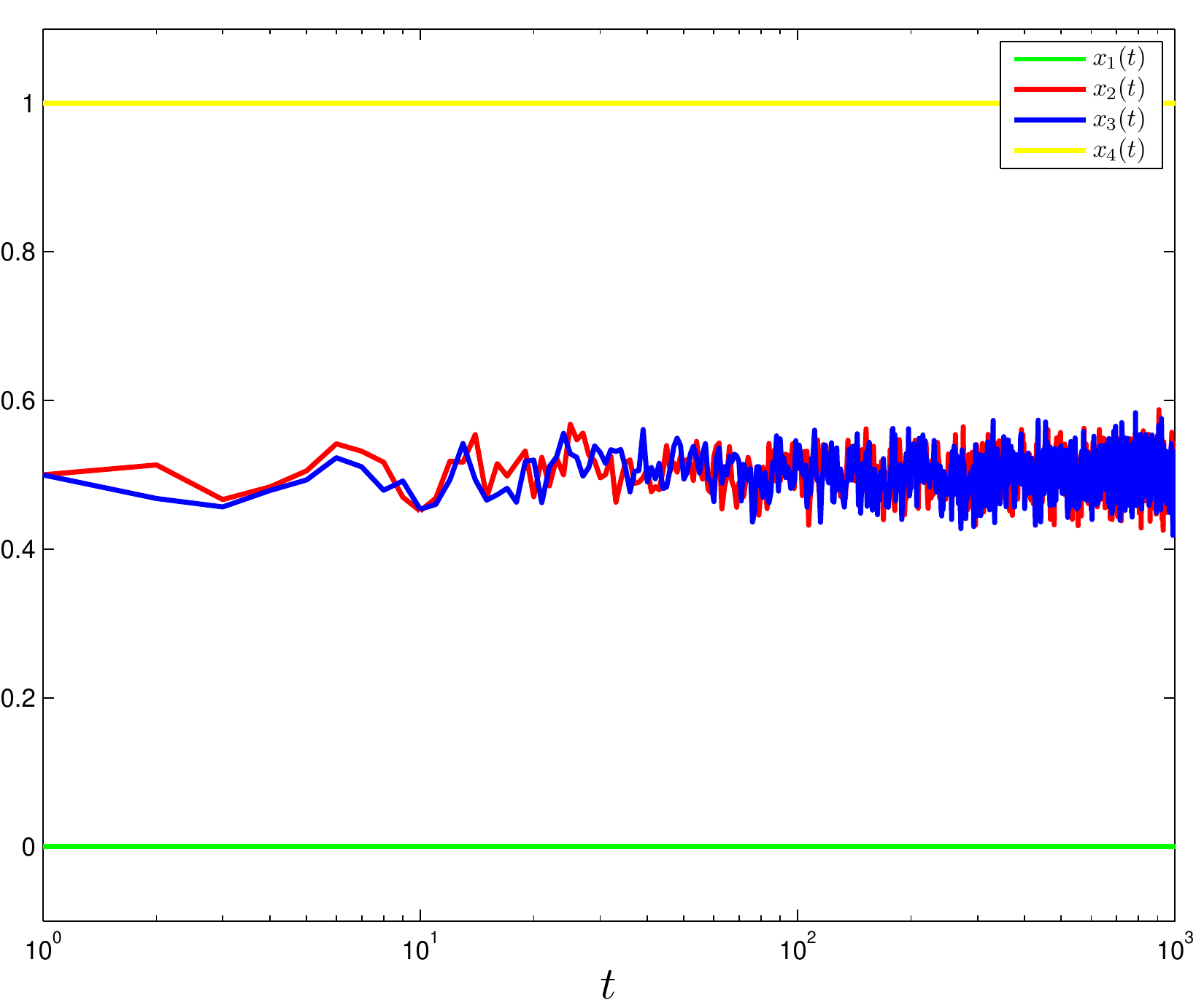}
\caption{The evolution of system (\ref{m4}) with $n=4$,
$x(0)=(0,1/2,1/2,1)$, $(r_1,\ldots,r_4)=(1/3,1,1,1/3)$, and $\eta=0.1$.} \label{EX46bn}
\end{center}
\end{figure}

\renewcommand{\thesection}{\Roman{section}}
\section{Conclusion and future work}\label{Conclusions}
\renewcommand{\thesection}{\arabic{section}}

The agreement and disagreement analysis of opinion dynamics has attracted an increasing amount of interest in recent years. On the other hand, all natural systems are inextricably affected by noise, and how noise affects the collective behavior of a complex system has also garnered considerable interest from researchers and developers in various fields. Thus, a natural problem is how the noise affects the agreement or bifurcation in opinion dynamics. This paper analyzes heterogeneous HK models with either environment or communication noise for the first time, and provides some critical results for quasi-synchronization.

There are still some problems that have not been considered. For example, because of limited space, this paper does not consider heterogeneous HK models with both environment and communication noises. Such systems may exhibit some interesting properties different from the systems (\ref{m1})-(\ref{m6}), though the analysis may be much more complex. These problems and models leave us with a direction for future research.

\renewcommand{\thesection}{\Roman{section}}

\appendices

\section{Proofs of Lemmas \ref{lem1}-\ref{lem4}}\label{Proof_Lemmas_a}

\begin{proof}[Proof of Lemma \ref{lem1}]
We consider the protocol (\ref{m1_new}) first.
Without loss of generality we assume $\alpha'\in(0,\eta/2)$.
 The main idea of this proof is: For each agent $i$, if its neighbors' average opinion $\widetilde{x}_i(t)$ is larger than an upper bound, we set $u_i(t)$ to be a negative input; if $\widetilde{x}_i(t)$ is less than a lower bound, we set $u_i(t)$ to be a positive input. Otherwise, we select a control input such that $x_i(t+1)$ will be in the interval $[z'-\alpha',z'+\alpha']$.
With this idea, for $t\geq 0$ and $1\leq i \leq n$  we choose $\delta_i(t)=\alpha'$ and
\begin{eqnarray}\label{lem1_1}
&&u_i(t)\\
&&=\left\{%
\begin{array}{ll}
-\eta+\alpha' & \mbox{if}~\widetilde{x}_i(t)> z'+\eta-\alpha',\\
z'-\widetilde{x}_i(t) & \mbox{if}~\widetilde{x}_i(t)\in [z'-\eta+\alpha',z'+\eta-\alpha'],\\
\eta-\alpha' & \mbox{if}~\widetilde{x}_i(t)< z'-\eta+\alpha'.
\end{array}%
\right.\nonumber
\end{eqnarray}
It can be computed that
\begin{eqnarray*}\label{lem1_2}
u_i(t)\in [-\eta+\delta_i(t),\eta-\delta_i(t)], ~~\forall i\in\mathcal{V}, t\geq 0,
\end{eqnarray*}
which meets its requirement in Definition \ref{def_reach}. Define
$$x_{\max}(t):=\max_{1\leq i \leq n}x_i(t)~~~~\mbox{and}~~~~x_{\min}(t):=\min_{1\leq i \leq n}x_i(t).$$
For any $i\in \mathcal{V}$, if $\widetilde{x}_i(t)>z'+\eta-\alpha'$ which indicates $x_{\max}(t)>z'+\eta-\alpha'$, by  (\ref{lem1_1})
we have
\begin{eqnarray}\label{lem1_3}
\begin{aligned}
&x_i(t+1)\\
&=\widetilde{x}_i(t)+u_i(t)+b_i(t)=\widetilde{x}_i(t)-\eta+\alpha'+b_i(t)\\
&\in (z'+b_i(t), x_{\max}(t)-\eta+\alpha'+b_i(t)]
\end{aligned}
\end{eqnarray}
and
\begin{eqnarray}\label{lem1_3a}
\begin{aligned}
&(z'+b_i(t), x_{\max}(t)-\eta+\alpha'+b_i(t)]\\
&\subseteq (z'-\alpha', x_{\max}(t)-\eta+2\alpha'].
\end{aligned}
\end{eqnarray}
If $\widetilde{x}_i(t)<z'-\eta+\alpha'$ which indicates $x_{\min}(t)<z'-\eta+\alpha'$,
similar to (\ref{lem1_3}) and (\ref{lem1_3a})  we have
\begin{eqnarray}\label{lem1_4}
\begin{aligned}
x_i(t+1)\in [x_{\min}(t)+\eta-2\alpha',z'+\alpha').
\end{aligned}
\end{eqnarray}
Otherwise if $\widetilde{x}_i(t)\in [z'-\eta+\alpha',z'+\eta-\alpha']$, by  (\ref{lem1_1})
we have
\begin{eqnarray}\label{lem1_5}
&&x_i(t+1)\\
&&=\widetilde{x}_i(t)+z'-\widetilde{x}_i(t)+b_i(t)\in [z'-\alpha',z'+\alpha'].\nonumber
\end{eqnarray}
From (\ref{lem1_3})-(\ref{lem1_5}) together we can get
\begin{eqnarray*}\label{lem1_6}
\begin{aligned}
x_{\max}(t+1)\leq \max\left\{x_{\max}(t)-\eta+2\alpha',z'+\alpha' \right\}
\end{aligned}
\end{eqnarray*}
and
\begin{eqnarray*}\label{lem1_7}
\begin{aligned}
x_{\min}(t+1)\geq \min \left\{x_{\min}(t)+\eta-2\alpha',z'-\alpha'  \right\},
\end{aligned}
\end{eqnarray*}
so there must exist a finite time $t_1$ such that $x_i(t_1)\in [z'-\alpha',z'+\alpha']$ for all $i\in \mathcal{V}$.

For protocol (\ref{m3_new}), this result can be obtained by a similar method to the above.
\end{proof}

\begin{proof}[Proof of Lemma \ref{lem2}]
Let $K$ be a constant satisfying
$$K\geq \max\Big\{ \frac{4(r_{\min}+\eta)}{2\eta-r_{\min}}, \frac{2n\eta}{n\eta-r_{\max}} \Big\}.$$
By Lemma \ref{lem1} the set $S_{\frac{K+2}{2K}r_{\min},\frac{r_{\min}}{K}}$ is  finite-time robustly reachable from $[0,1]^n$ under protocol
(\ref{m1_new}). We record the stop time of the set $S_{\frac{K+2}{2K}r_{\min},\frac{r_{\min}}{K}}$ being reached as $t_1$, which means
\begin{eqnarray*}\label{lem2_2}
x_i(t_1)\in [\frac{r_{\min}}{2}, \frac{r_{\min}}{2}+\frac{2r_{\min}}{K}],~~~~\forall i\in\mathcal{V}.
\end{eqnarray*}
 By this and (\ref{lem1_0}) we have
\begin{eqnarray}\label{lem2_3}
\begin{aligned}
&\widetilde{x}_1(t_1)=\widetilde{x}_2(t_1)=\cdots=\widetilde{x}_n(t_1)\\
&\in [\frac{r_{\min}}{2}, \frac{r_{\min}}{2}+\frac{2r_{\min}}{K}].
\end{aligned}
\end{eqnarray}

Without loss of generality we assume agent $1$ has the smallest interaction radius, i.e., $r_1=r_{\min}$.
For any $t\geq t_1$ and $i\in\mathcal{V}$ we choose $\delta_i(t)=\frac{\eta}{K}$, and
\begin{eqnarray}\label{lem2_4}
u_i(t)=\left\{%
\begin{array}{ll}
\eta-\frac{\eta}{K} & \mbox{if}~i=1,\\
-\eta+\frac{\eta}{K} & \mbox{otherwise}.
\end{array}%
\right.
\end{eqnarray}
Because $x_i(t_1+1)=\Pi_{[0,1]}(\widetilde{x}_i(t_1)+u_i(t_1)+b_i(t_1))$, by this and (\ref{lem2_3})  we can get
\begin{eqnarray}\label{lem2_5}
\begin{aligned}
x_1(t_1+1)&\geq \min \Big\{\widetilde{x}_1(t_1)+\eta-\frac{2\eta}{K},1\Big\}\\
&\geq \min \Big\{\frac{r_{\min}}{2}+\eta-\frac{2\eta}{K},1\Big\}>r_{\min},
\end{aligned}
\end{eqnarray}
and
\begin{eqnarray}\label{lem2_6}
\begin{aligned}
x_i(t_1+1)&\leq \max \Big\{\widetilde{x}_i(t_1)-\eta+\frac{2\eta}{K},0\Big\}\\
&\leq \max \Big\{\frac{r_{\min}}{2}+\frac{2r_{\min}}{K}-\eta+\frac{2\eta}{K},0\Big\}\\
&=0,~~~~i=2,\ldots,n,
\end{aligned}
\end{eqnarray}
so
\begin{eqnarray*}\label{lem2_7}
x_1(t_1+1)-x_i(t_1+1)>r_{\min}=r_1,~~i=2,\ldots,n.
\end{eqnarray*}
Next we compute $x_i(t_1+2)$. Because agent $1$ cannot receive information from the others at
time $t_1+1$, with (\ref{lem2_5}) we get
\begin{eqnarray}\label{lem2_8}
\begin{aligned}
x_1(t_1+2)&=x_1(t_1+1)+u_i(t_1+1)+b_i(t_1+1)\\
&\geq \min \Big\{\widetilde{x}_1(t_1)+2\Big(\eta-\frac{2\eta}{K}\Big),1\Big\}.
\end{aligned}
\end{eqnarray}
Also, for $2\leq i \leq n$, by (\ref{lem2_6}) we have
\begin{eqnarray*}\label{lem2_9}
\widetilde{x}_i(t_1+1)=\frac{x_1(t_1+1)}{n} I_{\{x_1(t_1+1)\leq r_i\}}\leq \frac{r_i}{n}\leq  \frac{r_{\max}}{n},
\end{eqnarray*}
which is followed by
\begin{eqnarray}\label{lem2_10}
\begin{aligned}
x_i(t_1+2)\leq \max \Big\{\frac{r_{\max}}{n}-\eta+\frac{2\eta}{K},0\Big\}=0.
\end{aligned}
\end{eqnarray}
Repeating the process of (\ref{lem2_8})-(\ref{lem2_10}) we get that there exists a finite time $t_2>t_1$ such that
$x_1(t_2)=1$, and $x_i(t_2)=0$ for $2\leq i \leq n$.
\end{proof}

\begin{proof}[Proof of Lemma \ref{lem3}]
We consider protocol (\ref{m1_new}) first. Without loss of generality we assume that $\varepsilon$ is  arbitrarily small (though
positive).
By Lemma \ref{lem1} the set $S_{\frac{1}{2},\frac{\varepsilon}{2}}$ is  finite-time robustly reachable from $[0,1]^n$ under protocol
(\ref{m1_new}). We record the stop time of the set $S_{\frac{1}{2},\frac{\varepsilon}{2}}$  being reached as $t_1$, which implies
\begin{eqnarray}\label{lem3_1}
\widetilde{x}_1(t_1)=\widetilde{x}_2(t_1)\in \Big[\frac{1}{2}-\frac{\varepsilon}{2},\frac{1}{2}+\frac{\varepsilon}{2}\Big].
\end{eqnarray}
We choose $\delta_i(t_1)=\frac{\varepsilon}{4}$ for any $i\in\mathcal{V}$,  and $u_1(t_1)=\eta-\frac{\varepsilon}{4}$,
and $u_2(t_1)=-\eta+\frac{\varepsilon}{4}$. With this and (\ref{lem3_1}) we have
\begin{eqnarray*}\label{lem3_2}
\begin{aligned}
x_1(t_1+1)&\geq \min\Big\{\widetilde{x}_1(t_1)+u_1(t_1)-\delta_1(t_1),1 \Big\}\\
&\geq \min\Big\{\widetilde{x}_1(t_1)+\eta-\frac{\varepsilon}{2},1 \Big\},
\end{aligned}
\end{eqnarray*}
and
\begin{eqnarray*}\label{lem3_3}
\begin{aligned}
x_2(t_1+1)&\leq \max\Big\{\widetilde{x}_2(t_1)+u_2(t_1)+\delta_2(t_1),0 \Big\}\\
&\leq \max\Big\{\widetilde{x}_2(t_1)-\eta+\frac{\varepsilon}{2},0 \Big\}.
\end{aligned}
\end{eqnarray*}
By these and (\ref{lem3_1}) we get $x_1(t_1+1)-x_2(t_1+1)\geq c_{\varepsilon}$.

For protocol (\ref{m3_new}), our result can be obtained by a similar method to the above.
\end{proof}

\begin{proof}[Proof of Lemma \ref{lem4}]
Without loss of generality we assume $r_1<2\eta$ and $\omega_1=\underline{\omega}_{\eta}$.
For any constant $K$ satisfying
\begin{eqnarray}\label{lem4_a00}
K\geq \max\Big\{ \frac{4(r_{1}+\eta)}{2\eta-r_{1}},\frac{4\eta n}{(n-1)w_1\varepsilon},\frac{2\eta n}{\varepsilon}\Big\},
 \end{eqnarray}
 with a similar method to the proof of Lemma \ref{lem2}
 we can find
a finite time $t_1$ such that
\begin{eqnarray}\label{lem4_00}
\widetilde{x}_1(t_1)=\widetilde{x}_2(t_1)=\cdots=\widetilde{x}_n(t_1)\in [\frac{r_{1}}{2}, \frac{r_{1}}{2}+\frac{2r_{1}}{K}].
\end{eqnarray}

For any $t\geq t_1$ and $i\in\mathcal{V}$ we choose $\delta_i(t)=\frac{\eta}{K}$, and $u_i(t)$
as same as (\ref{lem2_4}). Because $x_i(t_1+1)=\Pi_{[0,1]}(\widetilde{x}_i(t_1)+u_i(t_1)+b_i(t_1))$, similar to (\ref{lem2_5}) and (\ref{lem2_6}) we have
 \begin{eqnarray*}\label{lem4_1}
x_1(t_1+1)\geq \min \Big\{\widetilde{x}_1(t_1)+\eta-\frac{2\eta}{K},1\Big\}>r_{1},
\end{eqnarray*}
and
\begin{eqnarray*}\label{lem4_2}
\begin{aligned}
x_i(t_1+1)&\leq \max \Big\{\widetilde{x}_i(t_1)-\eta+\frac{2\eta}{K},0\Big\}=0,~i=2,\ldots,n.
\end{aligned}
\end{eqnarray*}
From these we get
\begin{eqnarray}\label{lem4_3}
&&x_1(t_1+2)\\
&&\geq  \min \Big\{\widetilde{x}_1(t_1+1)+\eta-\frac{2\eta}{K},1\Big\}\nonumber\\
&&=\min \Big\{\omega_1 x_{\rm{ave}}(t_1+1)+(1-\omega_1)x_1(t_1+1)\nonumber\\
&&~~~~~~~~~~+\eta-\frac{2\eta}{K},1\Big\}\nonumber\\
&&= \min \Big\{\Big(1-\omega_1+\frac{\omega_1}{n}\Big)x_1(t_1+1)+\eta-\frac{2\eta}{K},1\Big\},\nonumber
\end{eqnarray}
and for $2\leq i\leq n$ we have
\begin{eqnarray}\label{lem4_4}
\begin{aligned}
x_i(t_1+2) & \leq \max \Big\{\widetilde{x}_i(t_1+1) -\eta+\frac{2\eta}{K},0\Big\}\\
& \leq \max \Big\{\frac{x_1(t_1+1)}{n} -\eta+\frac{2\eta}{K},0\Big\}.
\end{aligned}
\end{eqnarray}
Also, for any $y< \min\{\frac{n\eta}{(n-1)\omega_1}-\varepsilon,n\eta-\varepsilon\}$, by (\ref{lem4_a00}) we have
\begin{eqnarray}\label{lem4_5}
\begin{aligned}
&\Big(1-\omega_1+\frac{\omega_1}{n}\Big)y+\eta-\frac{2\eta}{K}\\
&> y-\frac{(n-1)\omega_1}{n}\Big(\frac{n\eta}{(n-1)\omega_1}-\varepsilon\Big)+\eta-\frac{2\eta}{K}\\
&\geq y+\frac{(n-1)\omega_1\varepsilon}{2n}.
\end{aligned}
\end{eqnarray}
and
\begin{eqnarray}\label{lem4_6}
\begin{aligned}
\frac{y}{n} -\eta+\frac{2\eta}{K}< \frac{n\eta-\varepsilon}{n} -\eta+\frac{2\eta}{K}\leq 0.
\end{aligned}
\end{eqnarray}
Taking (\ref{lem4_5}) into (\ref{lem4_3}), and  (\ref{lem4_6}) into (\ref{lem4_4}), and repeatedly computing $x_i(t_1+3)$,  $x_i(t_1+4), \ldots$,
there must exist a finite time $t_2>t_1$ such that
\begin{eqnarray*}\label{lem4_7}
\begin{aligned}
x_1(t_2)-x_i(t_2)\geq \min\Big\{\frac{n\eta}{(n-1)\omega}-\varepsilon,n\eta-\varepsilon,1\Big\}
\end{aligned}
\end{eqnarray*}
for all $i\in\{2,\ldots,n\}$.
\end{proof}

\section{Proofs of Lemmas \ref{lem3_a}-\ref{lem4_a}}\label{Proof_Lemmas_b}

\begin{proof}[Proof of Lemma \ref{lem3_a}]
Without loss of generality we assume $\omega_1+\omega_2=\min_{i\neq j} (\omega_i+\omega_j)$, and  $\varepsilon$ is  arbitrarily small (though
positive).
Let $$x^*:=\Pi_{[0,1]}\left(\frac{1}{2}+\frac{(\omega_1-\omega_2)(n-1)\eta}{2n}\right).$$
By Lemma \ref{lem1_a} the set $S_{x^*,\frac{\varepsilon}{2}}$ is  finite-time robustly reachable from $[0,1]^n$ under protocol
(\ref{m6_new}). We record the stop time of the set $S_{x^*,\frac{\varepsilon}{2}}$  being reached as $t_1$, which implies
\begin{eqnarray}\label{lem3a_1}
\widetilde{x}_1(t_1)=\cdots=\widetilde{x}_n(t_1)\in \Big[x^*-\frac{\varepsilon}{2},x^*+\frac{\varepsilon}{2}\Big].
\end{eqnarray}
For $i=1,2$, let $\varepsilon_i:=\frac{n\varepsilon}{(1-\omega_i)(n-1)}$, and
choose $\delta_i(t_1)=\frac{\varepsilon_i}{4}$.
Also, choose $u_{j1}(t_1)=\eta-\frac{\varepsilon_1}{4}$ for any $j\in\mathcal{N}_1(t_1)\backslash\{1\}$
and $u_{j2}(t_1)=-\eta+\frac{\varepsilon_2}{4}$ for any $j\in\mathcal{N}_2(t_1)\backslash\{2\}$.
 Combining this with (\ref{m6_new}) and (\ref{lem3a_1}) we have
\begin{eqnarray*}\label{lem3a_2}
&&x_1(t_1+1)\\
&&=\Pi_{[0,1]}\Big(\widetilde{x}_1(t)+\frac{1-\omega_1}{n}\sum\limits_{j\neq 1}(u_{j1}(t)+b_{j1}(t))\Big)\\
&&\geq \min\Big\{\widetilde{x}_1(t_1)+\frac{(1-\omega_1)(n-1)}{n} [\eta-\frac{\varepsilon_1}{2}],1 \Big\}\\
&&\geq \min\Big\{x^*+\frac{(1-\omega_1)(n-1)\eta}{n}-\varepsilon,1 \Big\},
\end{eqnarray*}
and
\begin{eqnarray*}\label{lem3a_3}
&&x_2(t_1+1)\\
&&=\Pi_{[0,1]}\Big(\widetilde{x}_2(t)+\frac{1-\omega_2}{n}\sum\limits_{j\neq 2}(u_{j2}(t)+b_{j2}(t))\Big)\\
&&\leq \max \Big\{\widetilde{x}_2(t_1)+\frac{(1-\omega_2)(n-1)}{n} [-\eta+\frac{\varepsilon_2}{2}],0 \Big\}\\
&&\leq \max \Big\{x^*-\frac{(1-\omega_2)(n-1)\eta}{n} +\varepsilon,0 \Big\}.
\end{eqnarray*}
Thus we get
\begin{eqnarray}\label{lem3a_4}
&&x_1(t_1+1)-x_2(t_1+1)\nonumber\\
&&\geq \min \Big\{\frac{(2-\omega_1-\omega_2)(n-1)\eta}{n}-2\varepsilon,1 \Big\}\nonumber\\
&&= \min \Big\{a_{12}-2\varepsilon,1 \Big\}=c_{\varepsilon}.
\end{eqnarray}
\end{proof}

\begin{proof}[Proof of Lemma \ref{lem4_a}]
We first show that $E_{c_{12}-\varepsilon}$ is finite-time robustly reachable from $[0,1]^n$.
By a similar method to the proof of Lemma \ref{lem2}, for any real number $x^*\in[0,1]$,
 we can find
a finite time $t_1$ such that
\begin{eqnarray}\label{lem4a_00}
\begin{aligned}
&\widetilde{x}_1(t_1)=\widetilde{x}_2(t_1)=\cdots=\widetilde{x}_n(t_1)\\
&~~\in \Big[x^*-\frac{\eta}{K}, x^*+\frac{\eta}{K}\Big],
\end{aligned}
\end{eqnarray}
where $K>0$ is a constant large enough.

At the time $t_1$, we choose $\delta_1(t_1)=\delta_2(t_1)=\frac{\eta}{K}$, and $\delta_3(t_1)=\cdots=\delta_n(t_1)=\frac{\eta}{MK}$ with $M$ be a large integer.
 Also, for $i\in\mathcal{V}$ and $j\in\mathcal{N}_i(t_1)\setminus\{i\}$ we choose
\begin{eqnarray}\label{lem4a_01}
u_{ji}(t_1)=\left\{%
\begin{array}{ll}
\eta-\frac{2\eta}{K} & \mbox{if}~i=1\\
-\eta+\frac{2\eta}{K} & \mbox{if}~i=2\\
0 &\mbox{if}~3\leq i\leq n
\end{array}%
\right..
\end{eqnarray}
By (\ref{m6_new}), (\ref{lem4a_00}) and (\ref{lem4a_01}) we have
 \begin{eqnarray}\label{lem4a_02}
 \left\{
 \begin{aligned}
x_1(t_1+1)& \geq \min \Big\{\widetilde{x}_1(t_1)\\
&+\frac{(n-1)(1-\omega_1)}{n}\Big(\eta-\frac{3\eta}{K}\Big),1\Big\},\\
x_1(t_1+1)& \leq \min \Big\{\widetilde{x}_1(t_1)\\
&+\frac{(n-1)(1-\omega_1)}{n}\Big(\eta-\frac{\eta}{K}\Big),1\Big\},\\
x_2(t_1+1)& \geq \max \Big\{\widetilde{x}_1(t_1)\\
&-\frac{(n-1)(1-\omega_2)}{n}\Big(\eta-\frac{\eta}{K}\Big),0\Big\},\\
x_2(t_1+1)& \leq \max \Big\{\widetilde{x}_1(t_1)\\
&-\frac{(n-1)(1-\omega_2)}{n}\Big(\eta-\frac{3\eta}{K}\Big),0\Big\},\\
x_i(t_1+1)& \geq \min \Big\{\widetilde{x}_1(t_1)\\
&-\frac{(n-1)(1-\omega_i)\eta}{nMK} ,0\Big\}, ~3\leq i\leq n,\\
x_i(t_1+1)& \leq \max \Big\{\widetilde{x}_1(t_1)\\
&+\frac{(n-1)(1-\omega_i)\eta}{nMK},1\Big\}, ~3\leq i\leq n.
\end{aligned}
\right.
\end{eqnarray}
If $a_{12}=\frac{(n-1)(2-\omega_1-\omega_2)\eta}{n}\geq 1$,
similar to (\ref{lem3a_4})  we can choose suitable $x^*$ and large $K$ such that
\begin{eqnarray*}\label{lem4a_02_1}
\begin{aligned}
x_1(t_1+1)-x_2(t_1+1)&\geq \min \big\{a_{12}-\varepsilon,1 \big\}\\
&\geq 1-\varepsilon\geq c_{12}-\varepsilon,
\end{aligned}
\end{eqnarray*}
which indicates that $E_{c_{12}-\varepsilon}$ is robustly reached at time $t_1+1$.
Thus, we just need to consider the case
of  $a_{12}<1$. In this case we can choose suitable $x^*$ and large $K$ such that
\begin{eqnarray}\label{lem4a_03}
 \begin{aligned}
\widetilde{x}_1(t_1)+a_1<1 \mbox{~~and~~} \widetilde{x}_1(t_1)-a_2>0.
\end{aligned}
\end{eqnarray}
Here we recall that $a_i=\frac{(n-1)(1-\omega_i)\eta}{n}$. Also, we can get
\begin{eqnarray}\label{lem4a_04}
 \begin{aligned}
x_{\mbox{ave}}(t_1+1) \approx \widetilde{x}_1(t_1)+ \frac{(n-1)(\omega_2-\omega_1)\eta}{n^2},
\end{aligned}
\end{eqnarray}
where  $A\approx B$ indicates that $\lim_{K\rightarrow\infty}(A-B)=0$ in this proof.

At the time $t_1+1$, we choose $\delta_1(t_1+1)=\delta_2(t_1+1)=\frac{\eta}{K}$, and
$u_{j1}(t_1+1)=\eta-\frac{\eta}{K}$ for $j\in\mathcal{N}_1(t_1+1)\setminus\{1\}$, while $u_{j2}(t_1+1)=-\eta+\frac{\eta}{K}$ for $j\in\mathcal{N}_2(t_1+1)\setminus\{2\}$.

If $r_1<a_1$, by (\ref{lem4a_02}) we can get $\mathcal{N}_1(t_1+1)=\{1\}$ for large $K$, so by (\ref{m6_new})
\begin{eqnarray}\label{lem4a_1}
\begin{aligned}
&x_1(t_1+2)\\
&=\omega_1 x_{\rm{ave}}(t_1+1)+ (1-\omega_1)x_1(t_1+1)\\
&\approx  \widetilde{x}_1(t_1)+\frac{(n-1)\eta}{n}\Big[(1-\omega_1)^2+\frac{\omega_1(\omega_2-\omega_1)}{n}\Big].
\end{aligned}
\end{eqnarray}

If $r_1\in [a_1,a_{12})$, by (\ref{lem4a_02}) we can get $\mathcal{N}_1(t_1+1)=\{1,3,\ldots,n\}$ for large $K$ and $M$, so by (\ref{m6_new})
\begin{eqnarray}\label{lem4a_2}
\begin{aligned}
&x_1(t_1+2)\\
&\approx \min\Big\{\omega_1 x_{\rm{ave}}(t_1+1)+ \frac{1-\omega_1}{n-1}[(n-2)\widetilde{x}_1(t_1)\\
&~~~~~~~~~~+x_1(t_1+1)+(n-2)\eta],1\Big\}\\
&\approx \min\Big\{\widetilde{x}_1(t_1)+\frac{\omega_1(n-1)(\omega_2-\omega_1)\eta}{n^2}\\
&~~~~~~~~~~+(1-\omega_1)\Big[\frac{(1-\omega_1)\eta}{n}+\frac{(n-2)\eta}{n-1}\Big],1\Big\}.
\end{aligned}
\end{eqnarray}

If $r_1\geq a_{12}$, by (\ref{lem4a_02}) we can get $\mathcal{N}_1(t_1+1)=\{1,2,\ldots,n\}$ for large $K$ and $M$, so by   (\ref{m6_new})
and (\ref{lem4a_04}),
\begin{eqnarray}\label{lem4a_3}
&&x_1(t_1+2)\\
&&\approx \min\Big\{\omega_1 x_{\rm{ave}}(t_1+1)\nonumber\\
&&~~~~+ (1-\omega_1)[x_{\rm{ave}}(t_1+1)+\frac{(n-1)\eta}{n}],1\Big\}\nonumber\\
&&\approx \min\Big\{\widetilde{x}_1(t_1)+\frac{(n-1)\eta}{n}\Big(1-\omega_1+\frac{\omega_2-\omega_1}{n}\Big),1\Big\}.\nonumber
\end{eqnarray}

From (\ref{lem4a_1}), (\ref{lem4a_2}), (\ref{lem4a_3}) and (\ref{h_00})  we get
\begin{eqnarray}\label{lem4a_4}
\begin{aligned}
x_1(t_1+2)\approx \Pi_{[0,1]} (\widetilde{x}_1(t_1)+h_{12}).
\end{aligned}
\end{eqnarray}
By a similar discussion we can get
\begin{eqnarray}\label{lem4a_5}
\begin{aligned}
x_2(t_1+2)\approx \Pi_{[0,1]} (\widetilde{x}_1(t_1)-h_{21}).
\end{aligned}
\end{eqnarray}

Let $A\preceq B$ denote $\lim_{K\rightarrow\infty}(A-B) \leq 0$.
If $a_1\leq h_{12}$ and $a_2\leq h_{21}$, which means $x_1(t_1+1)\preceq x_1(t_1+2)$ and $x_2(t_1+2)\preceq x_2(t_1+1)$, then
we choose $x^*$ to be $h_{21}$ or $1-h_{12}$ and get $x_1(t_1+2)-x_2(t_1+2)\approx \min\{h_{12}+h_{21},1\}$, which indicates
 $E_{c_{12}-\varepsilon}$ is robustly reached at time $t_1+2$.

If $a_1> h_{12}$ and $a_2\leq h_{21}$,
we choose $x^*=1-a_1$ which means $x_1(t_1+1)\approx 1$,
 and so $x_1(t_1+2)-x_2(t_1+2)\approx \min\{h_{12}+h_{21},1-(a_1-h_{12})\}$, which indicates
 $E_{c_{12}-\varepsilon}$ is also robustly reached at time $t_1+2$.

If $a_1\leq h_{12}$ and $a_2>h_{21}$,
we choose $x^*=a_2$ and get $x_1(t_1+2)-x_2(t_1+2)\approx \min\{h_{12}+h_{21},1-(a_2-h_{21})\}$,
so $E_{c_{12}-\varepsilon}$ is also robustly reached at time $t_1+2$.

 If $a_1> h_{12}$ and $a_2>h_{21}$, we have
$h_{12}+h_{21}<x_1(t_1+1)-x_2(t_1+1)$, so  $E_{c_{12}-\varepsilon}$ is robustly reached at time $t_1+1$.

Given the discussion above,  $E_{c_{12}-\varepsilon}$ is finite-time robustly reachable from $[0,1]^n$.

For any $i\neq j$, by a similar method we have $E_{c_{ij}-\varepsilon}$ is finite-time robustly reachable from $[0,1]^n$,
so  $E_{c_\varepsilon}$ is finite-time robustly reachable from $[0,1]^n$.
\end{proof}

\section{Proofs of Lemmas \ref{lem1c}, \ref{lem2c}, and \ref{lem3c}}\label{Proof_Lemmas_c}

\begin{proof}[Proof of Lemma \ref{lem1c}]
Set $x_{\max}(t)$ and $x_{\min}(t)$ to be the maximal and minimum opinions at time $t$ respectively.
Let $K>0$ be a large constant and set $\delta_i(t)=\frac{\eta}{K}$ for $i\in\mathcal{V}$ and $t\geq 0$.
Next we try to find a control algorithm such that
\begin{eqnarray}\label{lem1c_1}
\begin{aligned}
&x_{\max}(t_1)-x_{\min}(t_1)\\
&<\max\Big\{ x_{\max}(0)-x_{\min}(0)- \frac{\eta}{2}+\frac{2\eta}{K},\\
&~~~~~~~~~~~x_{\max}(0)-x_{\min}(0)-r_1,\frac{2\eta}{K}\Big\},
\end{aligned}
\end{eqnarray}
where $t_1$ is a finite time. Assume $x_{\max}(0)-x_{\min}(0)>\frac{2\eta}{K}$.
Because for any initial opinions, there must exist two agents whose distance is not bigger than $\frac{1}{n-1}$, which means they are not isolated,
we prove (\ref{lem1c_1}) for the following three cases respectively:\\
Case I: If the agent with the minimum opinion is not isolated, for all $i\in\mathcal{V},j\in\mathcal{N}_i(0)\backslash\{i\}$ we choose
\begin{eqnarray}\label{lem1c_2}
u_{ji}(0)=\min\Big\{\eta,\frac{[x_{\max}(0)-\widetilde{x}_i(0)]|\mathcal{N}_i(0)|}{|\mathcal{N}_i(0)|-1}\Big\}-\frac{\eta}{K}.
\end{eqnarray}
By (\ref{m4_new}) and the fact that $x_{\min}(t)\leq \widetilde{x}_i(t)\leq x_{\max}(t)$,
we can get for any $i\in\mathcal{V}$ satisfying $|\mathcal{N}_i(0)|\geq 2$,
\begin{eqnarray*}\label{lem1c_4}
\begin{aligned}
x_{i}(1)& \geq \min\Big\{x_{\min}(0)+\frac{|\mathcal{N}_i(0)|-1}{|\mathcal{N}_i(0)|}\Big[\eta-\frac{2\eta}{K}\Big],\\
&~~~~~~~~~~~x_{\max}(0)-\frac{|\mathcal{N}_i(0)|-1}{|\mathcal{N}_i(0)|}\frac{2\eta}{K}\Big\}\\
&> \min\Big\{x_{\min}(0)+\frac{\eta}{2}-\frac{2\eta}{K},
x_{\max}(0)-\frac{2\eta}{K}\Big\}.
\end{aligned}
\end{eqnarray*}
Also, $x_i(1)\leq x_{\max}(0)$ for all $i\in\mathcal{V}$, and if $|\mathcal{N}_i(0)|=1$ then $x_i(0)>x_{\min}(0)+r_1$,
so (\ref{lem1c_1}) holds when $t_1=1$. \\
Case II: If the agent with the maximal opinion is not isolated, for all $i\in\mathcal{V},j\in\mathcal{N}_i(0)\backslash\{i\}$ we choose
\begin{eqnarray}\label{lem1c_5}
u_{ji}(0)=\max\Big\{-\eta,\frac{[x_{\min}(0)-\widetilde{x}_i(0)]|\mathcal{N}_i(0)|}{|\mathcal{N}_i(0)|-1}\Big\}+\frac{\eta}{K}.
\end{eqnarray}
Similar to Case I we get (\ref{lem1c_1}) holds when $t_1=1$. \\
Case III: If the agents with the minimum and maximal opinions are all isolated, for all $i\in\mathcal{V},j\in\mathcal{N}_i(0)\backslash\{i\}$ we choose
\begin{eqnarray}\label{lem1c_7}
u_{ji}(t)=\min\Big\{\eta,\frac{[x_{\max}(t)-\widetilde{x}_i(t)]|\mathcal{N}_i(t)|}{|\mathcal{N}_i(t)|-1}\Big\}-\frac{\eta}{K},
\end{eqnarray}
until the agent with the maximal opinion is not isolated. Let $y(t)$ be the minimal value of the non-isolated agents' opinions.
Under (\ref{lem1c_7}) we have
\begin{eqnarray*}\label{lem1c_8}
&&\min\Big\{y(t)+\frac{\eta}{2}-\frac{2\eta}{K},x_{\max}(t)-\frac{2\eta}{K}\Big\}\\
&&~ \leq y(t+1) \leq x_{\max}(t),~~~\forall i\in\mathcal{V},
\end{eqnarray*}
so there exists a finite time $t_0$ such that the agent with the maximal opinion is not isolated at time $t'$.
With the same method as Case II we can get (\ref{lem1c_1}) holds when $t_1=t_0+1$.

Repeatedly using (\ref{lem1c_1}) we get that there exists a finite time $t'$ such that $x_{\max}(t')-x_{\min}(t')\leq \frac{2\eta}{K}$.
Finally we design a control algorithm which moves $(x_{\min}(t')+x_{\max}(t'))/2$ to $z^*$ while $x_{\max}(t)-x_{\min}(t)$ keeps not bigger than
$\frac{2\eta}{K}$. Set $x'(t):=x_{\min}(t)+x_{\max}(t))/2$.
For any $t\geq t'$, $i\in\mathcal{V}$, and $j\in\mathcal{N}_i(t)\backslash\{i\}$  we choose
\begin{multline*}\label{lem1c_9}
u_{ji}(t)=\\
\left\{%
\begin{array}{ll}
\frac{[x'(t)-\widetilde{x}_i(t)]n}{n-1}-\eta+\frac{2\eta}{K} ~ \mbox{if}~x'(t)> z^*+\frac{n-1}{n}(\eta-\frac{2\eta}{K})\\
\frac{[z^*-\widetilde{x}_i(t)]n}{n-1} ~~~~~~~~~~~~~~~ \mbox{if}~\widetilde{x}_i(t)\in[z^*-\frac{n-1}{n}(\eta-\frac{2\eta}{K}),\\
~~~~~~~~~~~~~~~~~~~~~~~~~~~~~~~~~~~~~~~~z^*+\frac{n-1}{n}(\eta-\frac{2\eta}{K})]\\
\frac{[x'(t)-\widetilde{x}_i(t)]n}{n-1}+\eta-\frac{2\eta}{K} ~ \mbox{if}~x'(t)<z^*-\frac{n-1}{n}(\eta-\frac{2\eta}{K})
\end{array}%
\right..
\end{multline*}
By this and (\ref{m4_new}) we have $x'(t+1)\leq x'(t)-\frac{n-1}{n}(\eta-\frac{3\eta}{K})$ if $x'(t)> z^*+\frac{n-1}{n}(\eta-\frac{2\eta}{K})$,
and $x'(t+1)\geq x'(t)+\frac{n-1}{n}(\eta-\frac{3\eta}{K})$ if $x'(t)< z^*-\frac{n-1}{n}(\eta-\frac{2\eta}{K})$.
Then, there exists a finite time $t^*$ such that
\begin{eqnarray*}\label{lem1c_10}
\begin{aligned}
&x_i(t^*)\\
&=\widetilde{x}_i(t^*-1)+\frac{1}{n} \sum_{j\neq i}\Big(\frac{[z^*-\widetilde{x}_i(t^*-1)]n}{n-1}+b_{ji}(t^*-1)\Big)\\
&=z^*+\frac{1}{n}\sum_{j\neq i}b_{ji}(t^*-1),
\end{aligned}
\end{eqnarray*}
which indicates $|x_i(t^*)-z^*|<\frac{\eta}{K}$ for all $i\in\mathcal{V}$ by $|b_{ji}(t)|\leq \frac{\eta}{K}$. Thus,  $S_{z^*,\alpha^*}$ is finite-time robustly reachable from $[0,1]^n$ if we let $K\geq \lceil\eta/\alpha^*\rceil$.
\end{proof}

\begin{proof}[Proof of Lemma \ref{lem2c}]
Let $K>0$ be a large constant.
By Lemma \ref{lem1c},  the set $S_{\frac{1}{2},\frac{r_{1}}{K}}$ is  finite-time robustly reachable from $[0,1]^n$ under protocol
(\ref{m4_new}). We record the stop time of the set $S_{\frac{1}{2},\frac{r_{1}}{K}}$ being reached as $t_1$, which means
\begin{eqnarray}\label{lem2c_1}
x_i(t_1)\in \Big[\frac{1}{2}-\frac{r_{1}}{K}, \frac{1}{2}+\frac{r_{1}}{K}\Big],~~~~\forall i\in\mathcal{V}.
\end{eqnarray}
Then $|\mathcal{N}_i(t_1)|=n$ and $\widetilde{x}_i(t_1)=\sum_{j=1}^n \frac{x_j(t_1)}{n}$ for all $i\in\mathcal{V}$. Choose $\delta_i(t_1)=\frac{\eta}{K}$, and
\begin{eqnarray}\label{lem2c_2}
u_{ji}(t_1)=\left\{%
\begin{array}{ll}
-\eta+\frac{\eta}{K} & \mbox{if}~i=1, j\in\mathcal{N}_1(t)\backslash\{1\}\\
\eta-\frac{\eta}{K} & \mbox{if}~2\leq i\leq n, j\in\mathcal{N}_i(t)\backslash\{i\}
\end{array}%
\right.,
\end{eqnarray}
then  by (\ref{m4_new})  we have
\begin{eqnarray*}\label{lem2c_3}
x_1(t_1+1)\in\Big[\widetilde{x}_1(t_1)-\frac{n-1}{n}\eta, \widetilde{x}_1(t_1)-\frac{n-1}{n}\Big(\eta-\frac{\eta}{K}\Big)\Big],
\end{eqnarray*}
and
\begin{eqnarray*}\label{lem2c_4}
x_i(t_1+1)\in\Big[\widetilde{x}_i(t_1)+\frac{n-1}{n}\Big(\eta-\frac{\eta}{K}\Big),\widetilde{x}_i(t_1)+\frac{n-1}{n}\eta\Big]
\end{eqnarray*}
for all $i \in \{2,\ldots, n\}$.
Thus, $E_{\varepsilon}'$ is robustly reached at time $t_1+1$.

\end{proof}

\begin{proof}[Proof of Lemma \ref{lem3c}]
Similar to (\ref{lem2c_1}) there exists a time $t_1$ such that
\begin{eqnarray*}\label{lem3c_1}
x_i(t_1)\in \Big[\frac{r_1}{2}-\frac{r_{1}}{K}, \frac{r_1}{2}+\frac{r_{1}}{K}\Big],~~~~\forall i\in\mathcal{V},
\end{eqnarray*}
where $K>0$ is a large constant.
For all $t\geq t_1$ and $i\in\mathcal{V}$ we choose $\delta_i(t_1)=\frac{\eta}{K}$, and
\begin{eqnarray}\label{lem3c_2}
u_{ji}(t)=\left\{%
\begin{array}{ll}
-\eta+\frac{\eta}{K} & \mbox{if}~i=1, j\in\mathcal{N}_1(t)\backslash\{1\}\\
\eta-\frac{\eta}{K} & \mbox{if}~2\leq i\leq n, j\in\mathcal{N}_i(t)\backslash\{i\}
\end{array}%
\right.,
\end{eqnarray}
then
\begin{eqnarray*}\label{lem3c_3}
\widetilde{x}_i(t_1)=\sum_{j=1}^n \frac{x_j(t_1)}{n}\in \Big[\frac{r_1}{2}-\frac{r_{1}}{K}, \frac{r_1}{2}+\frac{r_{1}}{K}\Big],~~~~\forall i\in\mathcal{V}.
\end{eqnarray*}
Also, by (\ref{m4_new}) and the fact that $\eta>\frac{nr_1}{2(n-1)}$  we get
\begin{eqnarray*}\label{lem3c_3}
\begin{aligned}
x_1(t_1+1)& \leq \Pi_{[0,1]}\Big(\widetilde{x}_1(t_1)-\frac{n-1}{n}\Big(\eta-\frac{\eta}{K}\Big)\Big)\\
&\leq  \Pi_{[0,1]}\Big(\frac{r_1}{2}+\frac{r_1}{K}-\frac{n-1}{n}\Big(\eta-\frac{\eta}{K}\Big)\Big)=0,
\end{aligned}
\end{eqnarray*}
and for all $i\in\{2,\ldots,n\}$,
\begin{eqnarray*}\label{lem3c_4}
\begin{aligned}
x_i(t_1+1)\in\Big[\widetilde{x}_i(t_1)+\frac{n-1}{n}\Big(\eta-\frac{\eta}{K}\Big),\widetilde{x}_i(t_1)+\frac{n-1}{n}\eta\Big]
\end{aligned}
\end{eqnarray*}
and
\begin{eqnarray*}\label{lem3c_4a}
\begin{aligned}
&\Big[\widetilde{x}_i(t_1)+\frac{n-1}{n}\Big(\eta-\frac{\eta}{K}\Big),\widetilde{x}_i(t_1)+\frac{n-1}{n}\eta\Big]\\
&\subset \Big[\frac{r_1}{2}-\frac{r_1}{K}+\frac{n-1}{n}\Big(\eta-\frac{\eta}{K}\Big),\frac{r_1}{2}+\frac{r_1}{K}+\frac{n-1}{n}\eta\Big].
\end{aligned}
\end{eqnarray*}
These yield that $x_i(t_1+1)-x_1(t_1+1)>r_1$ for all $i=2,\ldots,n$, and $x_2(t_1+1),\cdots,x_n(t_1+1)$ are neighbors to each other.
Using (\ref{lem3c_2}) repeatedly, there exists a finite time $t_2>t_1$ such that $x_1(t_2)=0$ and $x_2(t_2)=\cdots=x_n(t_2)=1$.
\end{proof}

\begin{IEEEbiography}[{\includegraphics[scale=0.1]{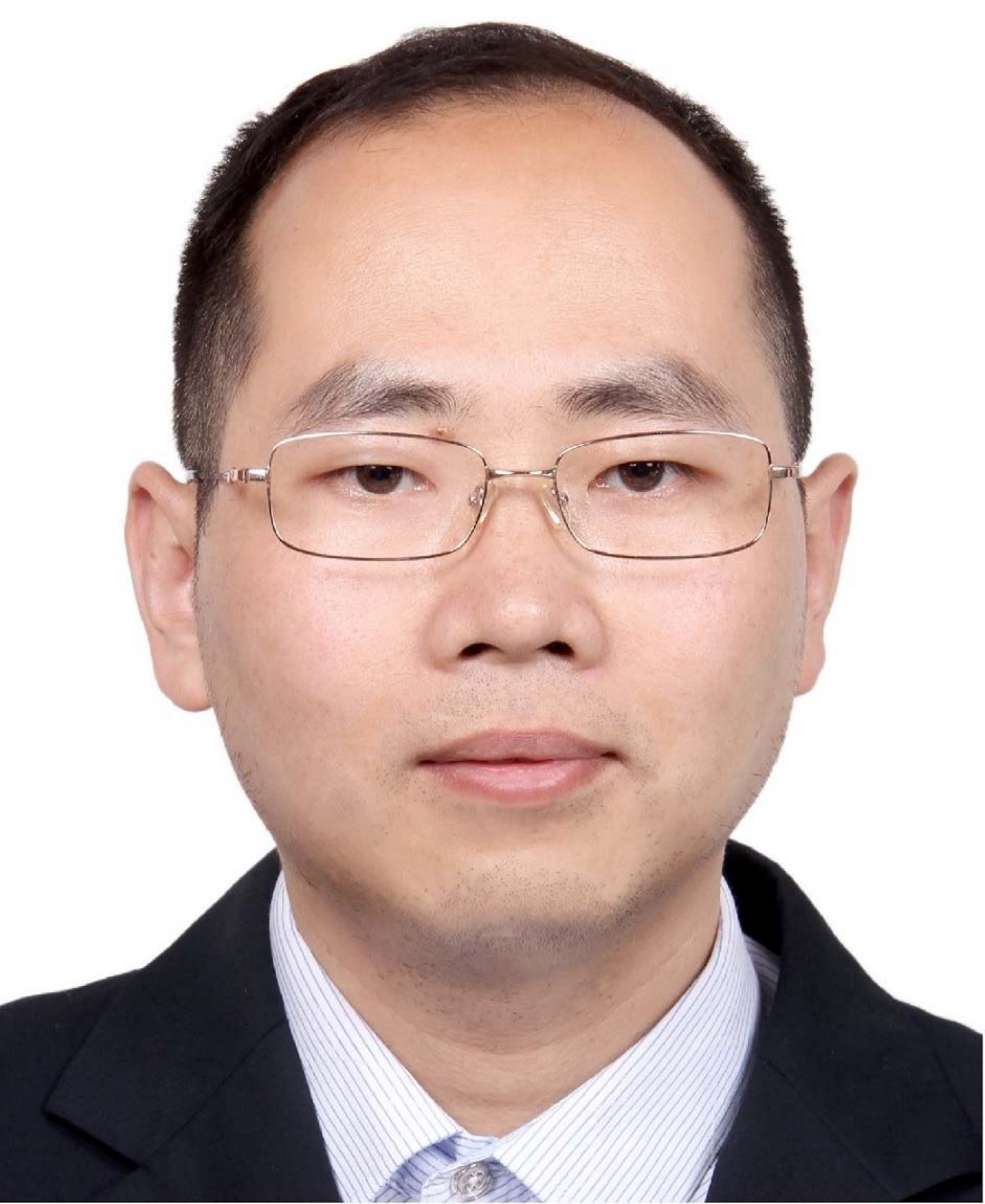}}]{Ge Chen} (M'13-SM'18)
 received the B.Sc. degree in mathematics from the University of Science and Technology of China in 2004, and the Ph.D. degree in mathematics from the University of Chinese Academy of Sciences, China, in 2009.

He jointed the  National Center for Mathematics and Interdisciplinary Sciences, Academy of Mathematics and Systems Science, Chinese Academy of Sciences in 2011, and is currently  an Associate Professor. His current research interest is the collective behavior of multi-agent systems.

Dr. Chen received the First Prize of the Application Award from the Operations Research Society of China (2010). One of his papers was selected as a SIGEST paper by the \emph{SIAM Review} (2014). He was also a finalist for the OR in Development prize from the International Federation of Operations Research Societies (2011).
\end{IEEEbiography}

\begin{IEEEbiography}[{\includegraphics[scale=0.15]{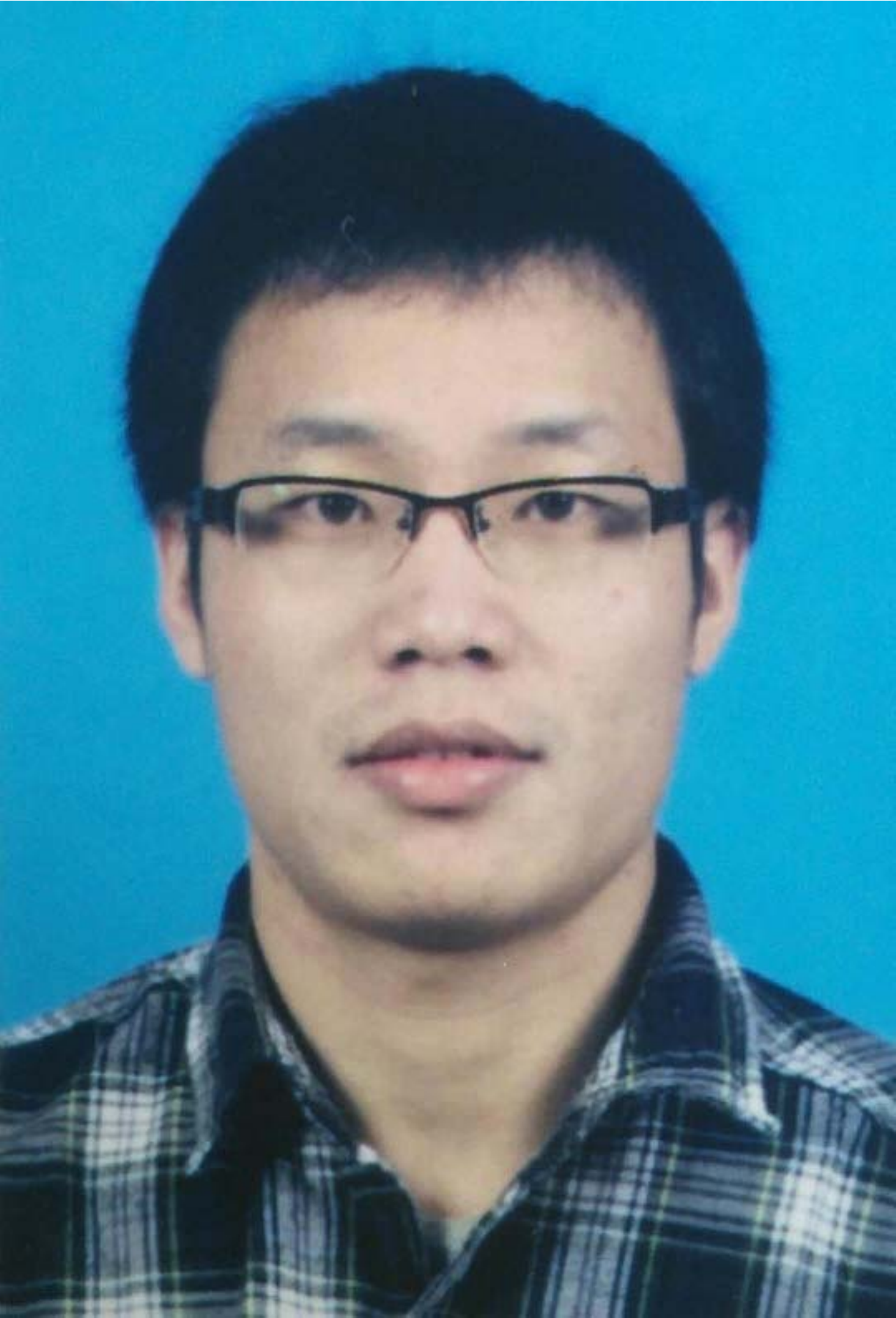}}]{Wei Su}
 received his B.S. degree from Beijing Jiaotong University (BJTU), China, in 2008, and his M.S. and PH.D. degrees form University of Chinese Academy of Sciences, China, in 2011 and 2015 respectively.
He is currently an assistant professor with School of Automation and Electrical Engineering, University of Science and Technology Beijing. His current research interests include social networks, machine learning, multi-agent systems and robustness of complex systems.
\end{IEEEbiography}

\begin{IEEEbiography}[{\includegraphics[scale=0.08]{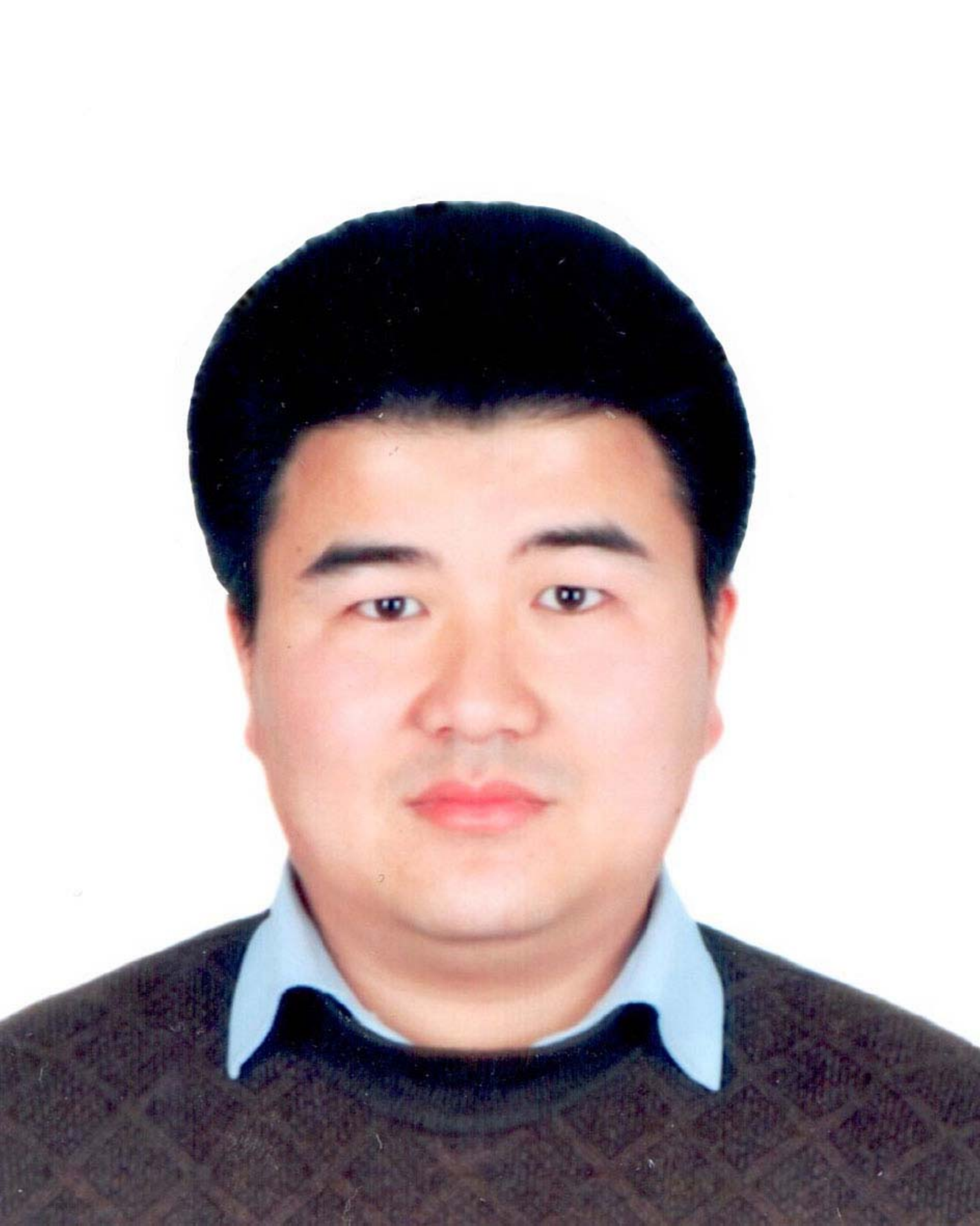}}]{Song-Yuan Ding}
 received his BSc in Chemical Physics at the University of Science and Technology of China in 2005 and his PhD in Chemistry under the supervision of Prof. Zhong-Qun Tian at Xiamen University in 2012. He is a Research Fellow in the Collaborative Innovation Center of Chemistry for Energy Materials(iChEM) at Xiamen University.Currently, his research interests include the control theory of catalyzed molecular assembly, surface-enhanced Raman spectroscopy for general materials, AFM-based infrared and Raman nanospectroscopy and imaging, and ab initial interfacial electrochemistry.
\end{IEEEbiography}

\begin{IEEEbiography}[{\includegraphics[scale=0.9]{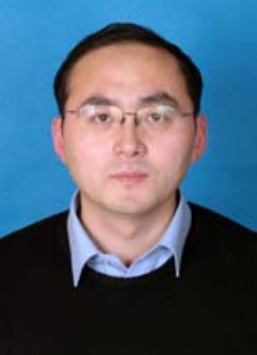}}]{Yiguang Hong} (M'99-SM'02-F'17)
 received his B.S. and M.S. degrees from Peking University, China, and the Ph.D. degree from the Chinese Academy of Sciences (CAS), China. He is currently a Professor in Academy of Mathematics and Systems Science, CAS, and serves as the Director of Key Lab of Systems and Control, CAS and the Director of the Information Technology Division, National Center for Mathematics and Interdisciplinary Sciences, CAS.  His current research interests include nonlinear control, multi-agent systems, distributed optimization/game, machine learning, and social networks.

Prof. Hong serves as Editor-in-Chief of Control Theory and Technology and Deputy Editor-in-Chief of Acta Automatica Sinca.   He also serves or served as Associate Editors for many journals, including the IEEE Transactions on Automatic Control, IEEE Transactions on Control of Network Systems, and IEEE Control Systems Magazine. He is a recipient of the Guang Zhaozhi Award at the Chinese Control Conference, Young Author Prize of the IFAC World Congress, Young Scientist Award of CAS, the Youth Award for Science and Technology of China, and the National Natural Science Prize of China.  He is a Fellow of IEEE.
\end{IEEEbiography}

\end{document}